\documentclass[12pt]{article}

\usepackage{amsmath,amssymb,amsthm}
\usepackage{enumerate,xspace,graphicx,fullpage}

\DeclareMathAlphabet\mathbfcal{OMS}{cmsy}{b}{n}

\numberwithin{equation}{section}

\newcommand{\Prob}[1]{\mathbf P\{#1\}}
\renewcommand{\P}{\mathbf P}
\newcommand{\one}{\mathbf 1}
\newcommand{\R}{\mathbb R}
\newcommand{\Sphere}{\mathbb{S}^{d-1}}
\newcommand{\cone}{\mathbb C}
\newcommand{\GG}{\mathbb G}
\newcommand{\KK}{\cone} 

\DeclareMathOperator{\clo}{cl}
\DeclareMathOperator{\co}{co}
\DeclareMathOperator{\supp}{supp}

\newcommand{\cco}{\overline{\mathrm{co}}\,}

\newtheorem{theorem}{Theorem}[section]
\newtheorem{proposition}[theorem]{Proposition}
\newtheorem{corollary}[theorem]{Corollary}
\newtheorem{lemma}[theorem]{Lemma}
\theoremstyle{definition} \newtheorem{definition}[theorem]{Definition}
\theoremstyle{remark}
\newtheorem{remark}[theorem]{Remark}
\newtheorem{example}[theorem]{Example}

\newcommand{\ve}{\mathtt{e}}
\newcommand{\vev}{\vec{\mathtt{e}}}
\newcommand{\ue}{\mathtt{u}}
\newcommand{\uev}{\vec{\mathtt{u}}}
\newcommand{\sA}{\mathcal{A}}
\newcommand{\sM}{\mathcal{M}}

\newcommand{\sP}{\mathcal{P}}
\newcommand{\sZ}{\mathcal{Z}}
\newcommand{\sF}{\mathcal{F}}

\newcommand{\E}{\mathbf{E}}
\newcommand{\Lp}[1][p]{L^{#1}}
\newcommand{\salg}{\mathfrak{F}}
\newcommand{\ssalg}{\mathfrak{H}}
\newcommand{\uE}{\mathbfcal{U}}
\newcommand{\vE}{\mathbfcal{E}}
\newcommand{\uvE}{\underline{\vE}}
\newcommand{\uuE}{\underline{\uE}}
\newcommand{\ovE}{\overline{\vE}}
\newcommand{\ouE}{\overline{\uE}}
\newcommand{\ouEnot}{\widetilde{\uE}}
\newcommand{\ovEnot}{\widetilde{\vE}}

\newcommand{\muv}{\mu}
\newcommand{\HX}[1][X]{H_\eta(#1)}
\newcommand{\HuX}[1][X]{H_u(#1)}
\newcommand{\HvX}[1][X]{H_v(#1)}

\newcommand{\sFC}{\co\sF(\cone)}

\newlength{\querylen}
\usepackage{fancybox}

\newcommand{\ssuperlinear}{superlinear\xspace}

\newcommand{\ssublinear}{sublinear\xspace}

\begin{document}


\title{Nonlinear expectations of random sets}

\author{Ilya Molchanov and Anja M\"uhlemann\\
{\small University of Bern, Institute of Mathematical Statistics}\\
{\small and Actuarial Science,
Alpeneggstrasse 22, CH-3012 Bern, Switzerland} }

\maketitle

\begin{abstract}
  Sublinear functionals of random variables are known as sublinear
  expectations; they are convex homogeneous functionals on
  infinite-dimensional linear spaces. We extend this concept for
  set-valued functionals defined on measurable set-valued functions
  (which form a nonlinear space), equivalently, on random closed
  sets. This calls for a separate study of \ssublinear and \ssuperlinear
  expectations, since a change of sign does not convert one to the
  other in the set-valued setting.

  We identify the extremal expectations as those arising from the
  primal and dual representations of them. Several general
  construction methods for nonlinear expectations are presented and
  the corresponding duality representation results are obtained. On
  the application side, sublinear expectations are naturally related
  to depth trimming of multivariate samples, while superlinear ones
  can be used to assess utilities of multiasset portfolios.
\end{abstract}




\section{Introduction}

Fix a probability space $(\Omega,\salg,\P)$. A \emph{sublinear}
expectation is a real-valued function $\ve$ defined on the space
$\Lp(\R)$ of $p$-integrable random variables (with $p\in[1,\infty]$),
such that
\begin{equation}
  \label{eq:5}
  \ve(\xi+a)=\ve(\xi)+a
\end{equation}
for each deterministic $a$, the function $\ve$ is monotone,
\begin{displaymath}
  \ve(\xi) \leq\ve(\eta)\quad \text{if } \xi \leq \eta\; \text{a.s.},
\end{displaymath}
homogeneous
\begin{displaymath}
  \ve(c\xi)=c\ve(\xi),\quad c\geq 0,
\end{displaymath}
and subadditive
\begin{equation}
  \label{eq:univ-sub}
  \ve(\xi+\eta)\leq \ve(\xi)+\ve(\eta), 
\end{equation}
see \cite{pen04}, who brought sublinear expectations to the
realm of probability theory and established their close relationship
to solutions of backward stochastic differential equations.  A
\emph{superlinear} expectation $\ue$ satisfies the same properties
with \eqref{eq:univ-sub} replaced by
\begin{equation}
  \label{eq:univ-sup}
  \ue(\xi+\eta)\geq \ue(\xi)+\ue(\eta).
\end{equation}
In many studies, the homogeneity property together with the sub-
(super-) additivity is replaced by the convexity of $\ve$ and the
concavity of $\ue$.  These nonlinear expectations may be defined on a
larger family than $\Lp$ or on its subfamily; it is necessary to
assume that the domain of definition contains all constants and is
closed under addition and multiplication by positive constants. The
range of values may be extended to $(-\infty,\infty]$ for the
sublinear expectation and to $[-\infty,\infty)$ for the superlinear
one. 

The choice of notation $\ve$ and $\ue$ is explained by the fact that
the superlinear expectation can be viewed as a utility function that
allocates a higher utility value to the sum of two random variables in
comparison with the sum of their individual utilities, see
\cite{delb12}. If random variable $\xi$ models a financial gain, then
$r(\xi)=-\ue(\xi)$ is called a \emph{coherent risk measure}. Property
\eqref{eq:5} is then termed cash invariance, and the superadditivity
property is turned into subadditivity due to the change of sign. The
subadditivity of risk means that the sum of two random variables bears
at most the same risk as the sum of their risks; this is justified by
the economic principle of diversification.

It is easy to see that $\ve$ is a sublinear expectation if and only if
\begin{equation}
  \label{eq:1}
  \ue(\xi)=-\ve(-\xi)
\end{equation}
is a superlinear one, and in this case $\ve$ and $\ue$ are said to
form an \emph{exact dual pair}. The sublinearity property yields that
\begin{displaymath}
  \ve(\xi)+\ve(-\xi)\geq \ve(0)=0\,,
\end{displaymath}
so that $-\ve(-\xi)\leq \ve(\xi)$. The interval $[\ue(\xi),\ve(\xi)]$
generated by an exact dual pair of nonlinear expectations
characterises the uncertainty in the determination of the expectation
of $\xi$. In finance, such intervals determine price ranges in
illiquid markets, see \cite{mad15}.  

We equip the space $\Lp$ with the $\sigma(\Lp,\Lp[q])$-topology based
on the standard pairing of $\Lp$ and $\Lp[q]$ with $1/p+1/q=1$.  It is
usually assumed that $\ve$ is lower semicontinuous and $\ue$ is upper
semicontinuous in the $\sigma(\Lp,\Lp[q])$-topology.
Given that $\ve$ and $\ue$ take finite values, general results of
functional analysis concerning convex functions on linear spaces imply
the semicontinuity property if $p\in[1,\infty)$ (see
\cite{kain:rues09}); it is additionally imposed if $p=\infty$.
A nonlinear expectation is said to be \emph{law invariant} (more
exactly, law-determined) if it takes the same value on identically
distributed random variables, see \cite[Sec.~4.5]{foel:sch04}.

A rich source of sublinear expectations is provided by suprema of
conventional (linear) expectations taken with respect to several
probability measures.  Assuming the $\sigma(\Lp,\Lp[q])$-lower
semicontinuity, the bipolar theorem yields that this is the only
possible case, see \cite{delb12} and \cite{kain:rues09}. Then
\begin{equation}
  \label{eq:supremum-vE}
  \ve(\xi)=\sup_{\gamma\in\sM,\E\gamma=1} \E(\gamma\xi)
\end{equation}
is the supremum of expectations $\E(\gamma\xi)$ over a convex
$\sigma(\Lp[q],\Lp)$-closed cone $\sM$ in $\Lp[q](\R_+)$; the
superlinear expectation is obtained by replacing the supremum with the
infimum. In the following, we assume that \eqref{eq:supremum-vE} holds
and the representing set $\sM$ is chosen in such a way to ensure that
the corresponding sublinear and superlinear expectations are law
invariant, that is, with each $\gamma$, $\sM$ contains all random
variables identically distributed as $\gamma$.

A \emph{random closed
  set} $X$ in Euclidean space is a random element with values in the
family $\sF$ of closed sets in $\R^d$ such that $\{X\cap
K\neq\emptyset\}$ is a measurable event for all compact sets $K$ in
$\R^d$, see \cite{mo1}. In other words, a random closed set is a
measurable \emph{set-valued function}. A random closed set $X$ is said
to be \emph{convex} if $X$ almost surely belongs to the family
$\co\sF$ of closed convex sets in $\R^d$. For convex random
sets in Euclidean space, the measurability condition is equivalent to
the fact that the support function of $X$ (see \eqref{eq:22}) is a
random function on $\R^d$ with values in $(-\infty,\infty]$. 

In the set-valued setting, it is natural to replace the inequalities
\eqref{eq:univ-sub} and \eqref{eq:univ-sup} with the inclusions. For
sets, the minus sign corresponds to the reflection with respect to the
origin; it does not alter the direction of the inclusion, and so there
is no direct link between set-valued sublinear and superlinear
expectations. Set inclusions are always considered nonstrict, e.g.,
$A\subset B$ allows for $A=B$.

This paper aims to systematically explore nonlinear set-valued
expectations. Section~\ref{sec:select-expect} recalls the classical
concept of the (linear) \emph{selection expectation} for random closed
sets, see \cite{aum65} and \cite[Sec.~2.1]{mo1}. A random
vector $\xi$ is said to be a \emph{selection} of $X$ if $\xi\in X$
almost surely.  The selection
expectation $\E X$ is defined as the closure of the set of
expectations of all integrable selections of $X$ (the primal
representation) or by considering the expected support function (being
the dual representation).
In this section, we introduce a suitable convergence concept for
(possibly, unbounded) random convex sets based on linear functionals
applied to the support function.

Nonlinear expectations of random convex sets are introduced in
Section~\ref{sec:general-non-linear}. The definitions refine the
properties of nonlinear expectations stated in
\cite[Sec.~2.2.7]{mo1}. Basic examples of such expectations and more
involved constructions are considered with a particular attention to
the expectations of random singletons and half-spaces. It is also
explained how the set-valued expectation applies to random convex
functions and how it is possible to get rid of the homogeneity
property and to extend the setting to convex/concave functionals.

Among the rather vast variety of nonlinear expectations, it is
possible to identify extremal ones: the \emph{minimal} sublinear
expectation of $X$ is the convex hull of nonlinear expectations of all
sets from some family that yields $X$ as their union. In the case of
selections, this becomes a direct generalisation of the primal
representation for the selection expectation. The \emph{maximal}
superlinear extension is the intersection of nonlinear expectations
of all half-spaces containing the random set. While in the linear case
the both coincide and provide two equivalent definitions of the
selection expectation, in general, the two constructions differ.

Nonlinear maps restricted to the family $\Lp(\R^d)$ of $p$-integrable
random vectors have been studied in \cite{cas:mol07,ham:hey10}, the
comprehensive duality results can be found in
\cite{drap:ham:kup16}. In our framework, these studies concern the
cases when the argument of a superlinear expectation is the sum of a
random vector and a convex cone. However, for general set-valued
arguments, it is not possible to rely the approach of
\cite{ham:hey10,drap:ham:kup16}, since the known techniques of
set-valued optimisation theory (see, e.g., \cite{khan:tam:z15}) are
not applicable.

The key technique suitable to handle nonlinear expectations relies on
the bipolar theorem. A direct generalisation of this theorem for
functional of random convex sets is not feasible, since random convex
sets do not form a linear space.  Section~\ref{sec:subl-expect}
provides duality results for sublinear expectations and
Section~\ref{sec:superl-expect-1} for the superlinear
ones. Specifically, the constant preserving minimal sublinear
expectations are identified.  For the superlinear expectation, the
family of random closed convex sets such that the sublinear
expectation contains the origin is a convex cone. However, it is
rather tricky to use the separation results, since linear functions
(such as the selection expectation) may have trivial values on
unbounded integrable random sets. For instance, the selection
expectation of a random half-space with a nondeterministic normal is
the whole space; in this case the superlinear expectation is not
dominated by any nontrivial linear one. In order to handle such
situations, the duality results for superlinear expectations are
proved for the maximal superlinear expectation. It is shown that the
superlinear expectation of a singleton is usually empty; in order to
come up with a nontrivial minimal extension, singletons in the
definition of the minimal extension are replaced by translated
cones. 

Some applications are outlined in Section~\ref{sec:applications}.
Sublinear expectations are useful as depth functions in order to
identify outliers in samples of random sets. Such samples often appear
in partial identified models in econometrics, see \cite{mol:mol16}.
The superlinear expectation is closely related to measuring
multivariate risk in finance and to multivariate utilities.
Superlinear expectations are useful to describe the utility, since the
utility of the sum of two portfolios described by random sets
``dominates'' the sum of their individual utilities. We show that the
minimal extension of a superlinear expectation is closely related to
the selection risk measure of lower random sets considered in
\cite{cas:mol14}.  

Appendix presents a self-contained proof of the fact that
vector-valued sublinear expectations of random vectors necessarily
split into sublinear expectations applied to each component of the
vector. This fact reiterates the point that the set-valued setting is
essential  for defining nonlinear expectations of random vectors. 

Note the following notational conventions: $X,Y$ denote random closed
convex sets, $F$ is a deterministic closed convex set, $\xi$ and
$\beta$ are $p$-integrable random vectors and random variables,
$\zeta$ and $\gamma$ are $q$-integrable vectors and variables with
$1/p+1/q=1$, $\eta$ is usually a random vector from the unit sphere
$\Sphere$, $u$ and $v$ are deterministic points from $\Sphere$.

\section{Selection expectation}
\label{sec:select-expect}

\subsection{Integrable random sets and selection expectation}
\label{sec:integr-rand-sets}

Let $X$ be a random closed set in $\R^d$, which is always assumed to
be almost surely non-empty.  A random vector $\xi$ is called a
\emph{selection} of $X$ if $\xi\in X$ almost surely. Let $\Lp(X)$
denote the family of $p$-integrable selections of $X$ for
$p\in[1,\infty)$, essentially bounded ones if $p=\infty$, and all
selections if $p=0$. If $\Lp(X)$ is not empty, then $X$ is called
\emph{$p$-integrable}, shortly integrable if $p=1$. This is the case
if $X$ is \emph{$p$-integrably bounded}, that is,
$\|X\|=\sup\{\|x\|:\; x\in X\}$ is $p$-integrable (essentially bounded
if $p=\infty$).

If $X$ is integrable, then its \emph{selection expectation} is
defined by 
\begin{equation}
  \label{eq:3}
  \E X=\clo\{\E \xi:\; \xi\in\Lp[1](X)\},
\end{equation}
which is the closure of the set of expectations of all integrable
selections of $X$, see \cite[Sec.~2.1.2]{mo1}. If $X$ is integrably
bounded, then the closure on the right-hand side is not needed, $\E X$
is compact, and also almost surely convex if $X$ is convex or the
underlying probability space is non-atomic. From now on, we assume
that all random closed sets are almost surely convex.

The \emph{support function} of a non-empty set $F$ in $\R^d$ is
defined by
\begin{equation}
  \label{eq:22}
  h(F,u)=\sup\{\langle x,u\rangle:\; x\in F\},\quad u\in\R^d,
\end{equation}
allowing for possibly infinite values if $F$ is not bounded, where
$\langle u,x\rangle $ denotes the scalar product. Due to homogeneity,
the support function is determined by its values on the unit sphere
$\Sphere$.

If $X$ is an integrable random closed set, then its expected support
function is the support function of $\E X$, that is, 
\begin{equation}
  \label{eq:16}
  \E h(X,u)=h(\E X,u),\quad u\in\R^d,
\end{equation}
see \cite[Th.~2.1.38]{mo1}. Thus,
\begin{displaymath}
  \E X=\bigcap_{u\in\Sphere} \{x:\; \langle x,u\rangle\leq \E h(X,u)\},
\end{displaymath}
which may be seen as the \emph{dual} representation of the selection
expectation with \eqref{eq:3} being its \emph{primal} representation.
\cite{arar:rud15} provide an axiomatic
Daniell--Stone type characterisation of the selection expectation. 
Property~\eqref{eq:16} can be also expressed as
\begin{equation}
  \label{eq:20}
  \E \sup_{\xi\in\Lp[1](X)} \langle\xi,u\rangle
  =\sup_{\xi\in\Lp[1](X)} \E \langle\xi,u\rangle,
\end{equation}
meaning that in this case it is possible to interchange the
expectation and the supremum. If
$X$ is an integrable random closed set and $\ssalg$ is a
sub-$\sigma$-algebra of $\salg$, the \emph{conditional expectation}
$\E(X|\ssalg)$ is identified by its support function, being the
conditional expectation of the support function of $X$, see
\cite{hia:ume77} and \cite[Sec.~2.1.6]{mo1}.

The dilation (scaling) of a closed set $F$ is defined as $cF=\{cx:\;
x\in F\}$ for $c\in\R$. 
For two closed sets $F_1$ and $F_2$, their 
\emph{closed Minkowski sum} is defined by
\begin{displaymath}
  F_1+F_2=\clo \{x+y:\;x\in F_1,\,y\in F_2\},
\end{displaymath}
and the sum is empty if at least one summand is empty. 
If at least one of $F_1$ and $F_2$ is compact, then the closure on the
right-hand side is not needed. We write shortly $F+a$ instead of
$F+\{a\}$ for $a\in\R^d$.

If $X$ and $Y$ are random closed convex sets, then $X+Y$ 
is a random closed set, see \cite[Th.~1.3.25]{mo1}.  
The selection expectation is \emph{linear} on integrable random
closed sets, that is,
\begin{displaymath}
  \E(X+Y)=\E X+ \E Y,
\end{displaymath}
see, e.g., \cite[Prop.~2.1.32]{mo1}. 

Let $\cone$ be a deterministic closed convex cone in $\R^d$ which is
distinct from the whole space. If $F=F+\cone$, then $F$ is said to be
$\cone$-closed. Due to the closed Minkowski sum on the right-hand side,
$F$ is also topologically closed. Let $\sFC$ denote the family of all
$\cone$-closed convex sets in $\R^d$ (including the empty set), and
let $\Lp(\sFC)$ be the family of all $p$-integrable random sets
with values in $\sFC$. Any random set from $\Lp(\sFC)$ is necessarily
a.s.\ non-empty.  
By
\begin{displaymath}
  \GG=\cone^o=\{u\in\R^d:\; h(\cone,u)\leq 0\}
\end{displaymath}
we denote the \emph{polar cone} to $\cone$. 

\begin{example}
  \label{ex:lower-sets}
  If $\cone=\{0\}$, then $\co\sF(\{0\})$ is the family of all convex
  closed sets in $\R^d$. If $\cone=\R_-^d$, then $\co\sF(\R_-^d)$ is
  the family of lower convex closed sets, and a random closed convex
  set with realisations in this family is called a random \emph{lower} set.
\end{example}

\begin{example}
  \label{ex:kabanov}
  Let $\cone$ be a convex closed cone in $\R^d$ which does not
  coincide with the whole space. If $X=\xi+\cone$ for
  $\xi\in\Lp(\R^d)$, then $X$ belongs to the space $\Lp(\sFC)$.
  For each $\zeta\in\Lp[q](\GG)$, we have $h(X,\zeta)=\langle
  \xi,\zeta\rangle$.
\end{example}

\subsection{Support function at random directions}
\label{sec:support-function-at}

Let
\begin{equation}
  \label{eq:38}
  H_u(t)=\{x\in\R^d:\; \langle x,u\rangle\leq t\}, \quad u\neq 0,
\end{equation}
denote a \emph{half-space} in $\R^d$, and let $H_u(\infty)=\R^d$.
Particular difficulties when dealing with \emph{unbounded} random
closed sets are caused by the fact that the support function of any
deterministic argument may be infinite with probability one.

\begin{example}
  \label{ex:non-h-approx}
  Let $X=H_\eta(0)$ be the random half-space with the normal vector
  $\eta$ having a non-atomic distribution. Then $\E X$ is the whole
  space. The support function of $X$ is finite only on the random ray
  $\{c\eta:\; c\geq0\}$.
\end{example}

It is shown in \cite[Cor.~3.5]{lep:mol17} that each random closed
convex set satisfies
\begin{equation}
  \label{eq:30}
  X=\bigcap_{\eta\in\Lp[0](\Sphere)} \HX,
\end{equation}
where 
\begin{displaymath}
  \HX=H_\eta(h(X,\eta))
\end{displaymath}
is the smallest half-space with outer normal $\eta$ that contains
$X$. If $X$ is a.s.\ $\cone$-closed, \eqref{eq:30} holds with $\eta$
running through the family of selections of $\Sphere\cap\GG$.

For each $\zeta\in\Lp[q](\R^d)$, the support function $h(X,\zeta)$ is
a random variable with values in $(-\infty,\infty]$, see
\cite[Lemma~3.1]{lep:mol17}. While $h(X,\zeta)$ is not necessarily
integrable, its negative part is always integrable if $X$ is
$p$-integrable. Indeed, choose any $\xi\in\Lp(X)$, and write
\begin{displaymath}
  h(X,\zeta)=h(X-\xi,\zeta)+\langle \xi,\zeta\rangle. 
\end{displaymath}
The second summand on the right-hand side is integrable, while the
first one is nonnegative. 

\begin{lemma}
  \label{lemma:sel-domination}
  Let $X,Y\in\Lp(\sFC)$. If $\E h(Y,\zeta)\leq \E h(X,\zeta)$ for all
  $\zeta\in\Lp[q](\GG)$, then $Y\subset X$ a.s. 
\end{lemma}
\begin{proof}
  For each measurable event $A$, replacing $\zeta$ with $\zeta\one_A$
  yields that 
  \begin{displaymath}
    \E[h(Y,\zeta)\one_A]\leq \E[h(X,\zeta)\one_A],
  \end{displaymath}
  whence $h(Y,\zeta)\leq h(X,\zeta)$ a.s. The same holds for a general
  $\zeta\in\Lp[q](\R^d)$ by splitting it into the cases when
  $\zeta\in\GG$ and $\zeta\notin\GG$. For a general
  $\zeta\in\Lp[0](\R^d)$, we have $h(Y,\zeta_n)\leq h(X,\zeta_n)$
  a.s. with $\zeta_n=\zeta\one_{\{\|\zeta\|\leq n\}}$, $n\geq1$. Thus,
  $h(Y,\zeta)\leq h(X,\zeta)$ a.s. for all $\zeta\in\Lp[0](\R^d)$, and
  the statement follows from \cite[Cor.~3.6]{lep:mol17}.
\end{proof}

\begin{corollary}
  \label{cor:distr-X}
  The distribution of $X\in\Lp(\sFC)$ is uniquely determined by $\E
  h(X,\zeta)$ for $\zeta\in\Lp[q](\GG)$.
\end{corollary}
\begin{proof}
  Apply Lemma~\ref{lemma:sel-domination} to $Y=\{\xi\}$, so that the
  values of $\E h(X,\zeta)$ identify all $p$-integrable selections of
  $X$, and note that $X$ equals the closure of the family of its
  $p$-integrable selections, see \cite[Prop.~2.1.4]{mo1}.
\end{proof}

A random closed set $X$ is called \emph{Hausdorff approximable} if it
appears as the almost sure limit in the Hausdorff metric of random
closed sets with at most a finite number of values. It is known
\cite[Th.~1.3.18]{mo1} that all random compact sets are Hausdorff
approximable, as well as those that appear as the sum of a random
compact set and a random closed set with at most a finite number of
possible values. The random closed set $X$ from
Example~\ref{ex:non-h-approx} is not Hausdorff approximable.

The distribution of a Hausdorff approximable $p$-integrable random
closed convex set $X$ is uniquely determined by the selection
expectations $\E(\gamma X)$ for all $\gamma\in\Lp[q](\R_+)$, actually
it suffices to let $\gamma$ be all measurable indicators, see
\cite{hes02} and \cite[Prop.~2.1.33]{mo1}.
If $X$ is Hausdorff approximable, then its selections $\xi$ are
identified by the condition $\E(\xi\one_A)\in \E(X\one_A)$ for all
events $A$. By passing to the support functions, we arrive at a
variant of Lemma~\ref{lemma:sel-domination} with $\zeta=u\one_A$ for
all $u\in\Sphere$ and $A\in\salg$.

\subsection{Convergence of random closed convex sets}
\label{sec:conv-rand-clos-1}

Convergence of random closed sets is typically considered in
probability, almost surely, or in distribution. In the following we
need to define $\Lp$-type convergence concepts suitable to deal with
unbounded random convex sets. 

The space $\Lp(\R^d)$ is equipped with the
$\sigma(\Lp,\Lp[q])$-topology, that is, $\xi_n\to \xi$ means that
$\E\langle\xi,\zeta\rangle\to\E\langle\xi,\zeta\rangle$ for all
$\zeta\in\Lp[q](\R^d)$.

\begin{lemma}
  \label{lemma:Lp-closed}
  If $X$ is a $p$-integrable random $\cone$-closed convex set, then
  $\Lp(X)$ is a non-empty convex $\sigma(\Lp,\Lp[q])$-closed and
  $\Lp(\cone)$-closed subset of $\Lp(\R^d)$. 
\end{lemma}
\begin{proof}
  If $\xi_n\in\Lp(X)$ and $\xi_n\to\xi\in\Lp(\R^d)$ in
  $\sigma(\Lp,\Lp[q])$, then
  \begin{displaymath}
    \E \langle\xi,\zeta\rangle=\lim \E\langle\xi_n,\zeta\rangle\leq \E
    h(X,\zeta)
  \end{displaymath}
  for all $\zeta\in\Lp[q](\R^d)$.  Thus, $\xi$ is a selection of $X$
  by Lemma~\ref{lemma:sel-domination}. The statement concerning
  $\cone$-closedness is obvious.
\end{proof}

A sequence $X_n\in\Lp(\sFC)$, $n\geq1$, is said to converge to
$X\in\Lp(\sFC)$ \emph{scalarly} in $\sigma(\Lp,\Lp[q])$ (shortly,
scalarly) if $\E h(X_n,\zeta)\to \E h(X,\zeta)$ for all
$\zeta\in\Lp[q](\GG)$, where the convergence is understood in the
extended line $(-\infty,\infty]$. Since $\E h(X_n,\zeta)$ equals the
support function of $\Lp(X_n)$ in direction $\zeta$, this convergence
is the scalar convergence $\Lp(X_n)\to\Lp(X)$ as convex sets in
$\Lp(\R^d)$, see \cite{son:zal92}.

\section{General nonlinear set-valued expectations}
\label{sec:general-non-linear}

\subsection{Definitions}
\label{sec:definitions}

Fix $p\in[1,\infty]$ and a convex closed cone $\cone$ distinct from
the whole space. 

\begin{definition}
  \label{def:vE}
  A \emph{sublinear set-valued} expectation is a
  function $\vE:\Lp(\sFC)\mapsto\co\sF$ such that:
  \begin{enumerate}[i)]
  \item for each deterministic $a\in\R^d$, 
    \begin{equation}
      \label{eq:39}
      \vE(X+a)=\vE(X)+a
    \end{equation}
    (additivity on deterministic singletons);
  \item $\vE(F)\supset F$ for all deterministic $F\in\sFC$;
  \item $\vE(X)\subset\vE(Y)$ if $X\subset Y$
    almost surely (monotonicity);
  \item $\vE(cX)=c\vE(X)$ for all $c>0$ (homogeneity);
  \item $\vE$ is subadditive, that is,
    \begin{equation}
      \label{eq:sub-X}
      \vE(X+Y)\subset \vE(X)+\vE(Y)
    \end{equation}
    for all $p$-integrable random closed convex sets $X$ and $Y$.
  \end{enumerate}
  A \emph{superlinear set-valued} expectation $\uE$ satisfies the same
  properties with the exception of ii) replaced by $\uE(F)\subset F$
  and \eqref{eq:sub-X} replaced by the superadditivity property
  \begin{equation}
    \label{eq:sup-X}
    \uE(X+Y)\supset \uE(X)+\uE(Y)\,.
  \end{equation}
  The nonlinear expectations $\vE$ and $\uE$ are said to be \emph{law
    invariant}, if they retain their values on identically distributed
  random closed convex sets.
\end{definition}

\begin{proposition}
  \label{prop:GG}
  Nonlinear expectations on $\Lp(\sFC)$ take values from $\sFC$. 
\end{proposition}
\begin{proof}
  If $a\in\cone$, then $X+a\subset X$ a.s., whence
  $\vE(X)+a\subset\vE(X)$. Therefore, $\vE(X)\in\sFC$. 
\end{proof}

While the argument $X$ of nonlinear expectations is a.s. non-empty,
$\uE(X)$ may be empty and then the right-hand side of \eqref{eq:sup-X}
is also empty. However, if $\vE(X)$ is empty for some $X$, then
$\vE(\xi+\cone)=\emptyset$ for $\xi\in\Lp(X)$, hence
\begin{displaymath}
  \vE(Y)=\vE(Y+\cone)=\vE(Y-\xi+\xi+\cone)\subset
  \vE(Y-\xi)+\vE(\xi+\cone)=\emptyset
\end{displaymath}
is empty for all $p$-integrable random sets $Y$. In view of this, it
is assumed that sublinear expectations take non-empty values. We
always exclude the trivial cases, when $\vE(X)=\R^d$ or
$\uE(X)=\emptyset$ for all $X$.

The homogeneity property immediately implies that $\vE(X)$ and
$\uE(X)$ are cones if $X$ is almost surely a cone, that is, $cX=X$
a.s. for all $c>0$.  Therefore, it is only possible to conclude that
$\vE(\cone)$ is a closed convex cone, which may be strictly larger
than $\cone$.  By Proposition~\ref{prop:GG}, $\uE(\cone)$ is either
$\cone$ or is empty. 

The sublinear (respectively, superlinear) expectation is said to be
\emph{normalised} if $\vE(\cone)=\cone$ (respectively,
$\uE(\cone)=\cone$). We always have $\vE(\R^d)=\R^d$ by property ii),
and also $\uE(\R^d)=\R^d$, since $\uE(\R^d)=\uE(\R^d)+a$ for all
$a\in\R^d$, and $\uE$ is not identically empty.

The properties of the nonlinear expectations do not imply that they
preserve deterministic convex closed sets.  The family
$\{F\in\sFC:\; \vE(F)=F\}$ of invariant sets is closed under
translations, dilations by positive reals, and for Minkowski sums,
since if $\vE(F)=F$ and $\vE(F')=F'$, then
\begin{displaymath}
  F+F'\subset \vE(F+F')\subset \vE(F)+\vE(F')=F+F'.
\end{displaymath}
A nonlinear expectation is said to be \emph{constant preserving} if 
all non-empty deterministic sets from $\sFC$ are invariant. 


The superlinear and sublinear expectations form a \emph{dual pair}
if\; $\uE(X)\subset\vE(X)$ for each $p$-integrable random closed
convex set $X$.  In difference to the univariate setting, the exact
duality relation \eqref{eq:1} is useless; if $\cone=\{0\}$, then
$-\vE(-X)$ is also a sublinear expectation, where $-X=\{-x:\; x\in
X\}$ is the reflection of $X$ with respect to the origin.

For a sequence $\{F_n,n\geq1\}$ of closed sets, its \emph{lower
  limit}, $\liminf F_n$, is the set of limits for all convergent
sequences $x_n\in F_n$, $n\geq1$, and its \emph{upper limit}, $\limsup
F_n$, is the set of limits for all convergent subsequences $x_{n_k}\in
F_{n_k}$, $k\geq1$.

The sublinear expectation $\vE$ is called \emph{lower semicontinuous}
if
\begin{equation}
  \label{eq:2}
  h(\vE(X),u)\subset \liminf h(\vE(X_n),u),\quad u\in\R^d,
\end{equation}
and $\uE$ is \emph{upper semicontinuous} if 
\begin{displaymath}
  \uE(X)\supset \limsup\, \uE(X_n)
\end{displaymath}
for a sequence of random closed convex sets $\{X_n,n\geq1\}$
converging to $X$ in the chosen topology, e.g. scalarly lower
semicontinuous if $X_n$ scalarly converges to $X$. Note that the lower
semicontinuity definition is weaker than its standard variant for
set-valued functions that would require that $\vE$ is a subset of
$\liminf \vE(X_n)$, see \cite[Prop.~2.35]{hu:pap97}.

\begin{remark}
  \label{rem:restriction}
  It is possible to consider nonlinear expectations defined only on
  some special random sets, e.g., singletons or half-spaces. It is
  only required that the family of such sets is closed under
  translations, dilations by positive reals, and for Minkowski sums.

  The family $\co\sF$ is often ordered by the \emph{reverse inclusion}
  ordering; then the terminology is correspondingly adjusted, e.g.,
  the superlinear expectation becomes sublinear. However, we
  systematically consider the conventional inclusion order.
\end{remark}

\begin{remark}
  \label{rem:convex}
  Motivated by financial applications, it is possible to replace the
  homogeneity and sub- (super-) additivity properties with 
  convexity or concavity, e.g.,
  \begin{displaymath}
    \uE(\lambda X+(1-\lambda)Y)\supset
    \lambda\;\uE(X)+(1-\lambda)\;\uE(Y),
    \quad \lambda\in[0,1].
  \end{displaymath}
  However, then $\uE$ can be turned into a superlinear expectation
  $\uE'$ for random sets in the space $\R^{d+1}$ by letting
  \begin{displaymath}
    \uE'(\{t\}\times X)=\{t\}\times t\;\uE(t^{-1}X), \quad t>0.
  \end{displaymath}
  The arguments of $\uE'$ are random closed convex sets
  $Y=\{t\}\times X$; they form a family closed for dilations, 
  Minkowski sums and translations by singletons from
  $\R_+\times\R^d$. Note that selections of $\{t\}\times X$ are given
  by $(t,\xi)$ with $\xi$ being a selection of $X$. In view of this,
  all results in the homogeneous case apply to the convex case if
  dimension is increased by one.
\end{remark}

\subsection{Examples}
\label{sec:examples}

The simplest example is provided by the selection expectation, which
is linear and law invariant on all integrable random convex sets. 

\begin{example}[Fixed points and support]
  \label{ex:fp-nonlinear}
  Let 
  \begin{displaymath}
    F_X=\big\{x:\; \Prob{x\in X}=1\big\}
  \end{displaymath}
  denote the set of \emph{fixed points} of a random closed set $X$.
  If $X$ is almost surely convex, then $F_X$ is also almost surely
  convex, and if $X$ is compact with a positive probability, then
  $F_X$ is compact. It is easy to see that $F_{X+Y}\supset F_X+F_Y$,
  whence $\uE(X)=F_X$ is a law invariant superlinear expectation. With
  a similar idea, it is possible to define the sublinear expectation
  $\vE(X)=\supp X$ as the \emph{support} of $X$, which is the set of
  points $x$ such that $X$ hits any open neighbourhood of $x$ with a
  positive probability. By the monotonicity property,
  $\{x\}=\uE(\{x\})\subset \uE(X)$ for any $x\in F_X$, whence
  $\uE(X)=F_X$ is a subset of any other normalised superlinear
  expectation of $X$. By a similar argument, $\vE(X)=\supp X$
  dominates any other constant preserving sublinear expectation. 
\end{example}

\begin{example}[Half-lines]
  \label{ex:h-line}
  Fix $\cone=\{0\}$ and let $X=[\xi,\infty)\subset \R$. Then
  $\uE(X)=[\ve(\xi),\infty)$ is superlinear if and only if $\ve(\xi)$
  is sublinear in the usual sense of \eqref{eq:univ-sub}. For random
  sets of the type $Y=(-\infty,\xi]$, the superlinearity of
  $\uE(Y)=(-\infty,\ue(\xi)]$ corresponds to the univariate
  superlinearity of $\ue(\xi)$. Therefore, the nature of a set-valued
  nonlinear expectation depends not only on the background numerical
  one, but also on the construction of relevant random sets. The
  situation becomes more complicated in higher dimensions, where
  complements of convex sets are not necessarily convex and the
  Minkowski sum of complements is not equal to the complement of the
  sum.
\end{example}

\begin{example}[Random intervals]
  \label{ex:rand-intervals}
  Let $X=[\eta,\xi]$ be a random interval on the line with
  $\xi,\eta\in\Lp(\R)$, and let $\cone=\{0\}$. Then
  $\vE(X)=[\ue(\eta),\ve(\xi)]$ is the interval formed by a numerical
  superlinear expectation of $\eta$ and a numerical sublinear
  expectation of $\xi$ such that $\ue(\xi)\leq \ve(\xi)$ for all
  $\xi$, e.g., if $\ue$ and $\ve$ are an exact dual pair. The
  superlinear expectation $\uE(X)=[\ve(\eta),\ue(\xi)]$ may be empty.
\end{example}


\subsection{Expectations of singletons and half-spaces}
\label{sec:transl-equiv}

The additivity property on deterministic singletons immediately yields
the following useful fact.

\begin{lemma}
  \label{lemma:primal}
  We have $\vE(X)=\{x\in\R^d:\; \vE(X-x)\ni 0\}$, and the same holds
  for the superlinear expectation. 
\end{lemma}

Fix $\cone=\{0\}$.  Restricted to singletons, the sublinear
expectation is a homogeneous map $\vE:\Lp(\R^d)\mapsto\co\sF$ that
satisfies
\begin{displaymath}
  \vE(\{\xi+\eta\})\subset \vE(\{\xi\})+\vE(\{\eta\}),\quad
  \xi,\eta\in\Lp(\R^d). 
\end{displaymath}
Note that $\vE(\{\xi\})$ is not necessarily a singleton.  If
$\vE(\{\xi\})$ is a singleton for each $\xi\in\Lp(\R^d)$, then $\vE$
is linear on $\Lp(\R^d)$. Assuming its lower semicontinuity, it
becomes the usual (linear) expectation. The following result concerns
the superlinear expectation of singletons. For a general cone $\cone$,
a similar result holds with singletons replaced by sets $\xi+\cone$. 

\begin{proposition}
  \label{prop:singl-sup}
  Let $\cone=\{0\}$.  For each $\xi\in\Lp(\R^d)$ and any normalised
  superlinear expectation $\uE$, the set $\uE(\{\xi\})$ is either
  empty or a singleton, and $\uE$ is additive on the family of all
  singletons with non-empty $\uE(\{\xi\})$.
\end{proposition}
\begin{proof}
  By \eqref{eq:sup-X} applied to $X=\{\xi\}$ and $Y=\{-\xi\}$, we
  have 
  \begin{displaymath}
    \{0\}= \uE(\{0\})\supset \uE(\{\xi\})+\uE(\{-\xi\}),
  \end{displaymath}
  whence $\uE(\{\xi\})$ is either empty or is a singleton, and then
  $\uE(\{-\xi\})=-\uE(\{\xi\})$. If $\uE(\{\xi\})$ and $\uE(\{\xi'\})$
  are singletons (and so are non-empty) for $\xi,\xi'\in\Lp(\R^d)$,
  then
  \begin{displaymath}
    \uE(\{\xi+\xi'\})\supset \uE(\{\xi\})+\uE(\{\xi'\}),
  \end{displaymath}
  whence the inclusion turns into the equality. 
\end{proof}

In view of Proposition~\ref{prop:singl-sup} and imposing the upper
semicontinuity property on the superlinear expectation, $\uE(\{\xi\})$
equals $\{\E \xi\}$ or is empty for each $p$-integrable $\xi$. The
family of $\xi\in\Lp(\R^d)$ such that $\uE(\{\xi\})\neq\emptyset$ is a
convex cone in $\Lp(\R^d)$.

\begin{proposition}
  \label{prop:whole-space}
  If $X+X'=\R^d$ a.s. for $X'$ being an independent copy of $X$, then
  $\vE(X)=\R^d$ for each law invariant sublinear expectation $\vE$. 
\end{proposition}
\begin{proof}
  By subadditivity and law invariance,
  \begin{displaymath}
  \R^d=\vE(\R^d)=\vE(X+X')\subset \vE(X)+\vE(X')=2 \vE(X). \qedhere
\end{displaymath}
\end{proof}

Proposition~\ref{prop:whole-space} applies if $X=H_\eta(0)$ is a
half-space with a non-atomic $\eta$, so that each law invariant
sublinear expectation on such random sets takes trivial values.

\begin{example}
  \label{ex:lower-quadrant}
  Let $\cone=\R_-^d$. If $\vE(\xi+\R_-^d)=\vev(\xi)+\R_-^d$ for a
  vector-valued function $\vev:\Lp(\R^d)\mapsto\R^d$, then $\vev(\xi)$
  splits into the vector of superlinear expectations applied to the
  components of $\xi=(\xi_1,\dots,\xi_d)$, see
  Theorem~\ref{thr:split}. 
\end{example}

\subsection{Nonlinear expectations of random convex functions}
\label{sec:nonl-expect-rand}

A lower semicontinuous convex function $f:\R^d\mapsto[0,\infty]$
yields a convex set $T_f$ in $\R^{d+1}$ such that
\begin{displaymath}
  h(T_f,(t,x))=
  \begin{cases}
    t f(x/t), & t>0,\\
    0, & \text{otherwise}.
  \end{cases}
\end{displaymath}
The obtained support function is called the \emph{perspective
  transform} of $f$, see \cite{hir:lem93}. Note that $f$ can be
recovered by letting $t=1$ in the support function of $T_f$. 

If $\xi(x)$, $x\in\R^d$, is a random nonnegative lower semicontinuous
convex function, then its sublinear expectation can be defined as
$\vE(\xi)(x)=h(\vE(T_\xi),(1,x))$, and the superlinear one is defined
similarly. With this definition, all constructions from this paper
apply to random functions.

\section{Extensions of nonlinear expectations}
\label{sec:monotonicity}

\subsection{Minimal extension}
\label{sec:minimal-extension}

The \emph{minimal extension} of a sublinear set-valued expectation
$\vE$ on random sets from $\Lp(\sFC)$ is defined by
\begin{equation}
  \label{eq:13}
  \uvE(X)=\cco \bigcup_{\xi\in\Lp(X)} \vE(\xi+\cone),
\end{equation}
where $\cco$ denotes the closed convex hull operation.  It extends a
sublinear expectation defined on sets $\xi+\cone$ to all
$p$-integrable random closed sets $X$ such that $X=X+\cone$ a.s.  In
terms of support functions, the minimal extension is given by
\begin{equation}
  \label{eq:23}
  h(\uvE(X),u)=\sup_{\xi\in\Lp(X)} h(\vE(\xi+\cone),u),
  \quad u\in\GG.
\end{equation}

\begin{proposition}
  \label{prop:sub-extension}
  If $\vE$ is a sublinear expectation defined on random sets
  $\xi+\cone$ for $\xi\in\Lp(\R^d)$, then its minimal extension
  \eqref{eq:13} is a sublinear expectation.
\end{proposition}
\begin{proof}
  The additivity of $\uvE$ on deterministic singletons follows from
  this property of $\vE$. For a deterministic $F\in\co\sFC$,
  \begin{displaymath}
    \uvE(F) \supset \cco \bigcup_{x\in F} \vE(x+\cone)
    \supset \cco \bigcup_{x\in F} (x+\cone)=F.
  \end{displaymath}
  The homogeneity and monotonicity properties of $\uvE$ are
  obvious. The subadditivity follows from the fact that $\Lp(X+Y)$ is
  the $\Lp$-closure of the sum $\Lp(X)+\Lp(Y)$, see
  \cite[Prop.~2.1.6]{mo1}. 
\end{proof}

\subsection{Maximal extension}
\label{sec:maximal-extension}

Extending a superlinear expectation $\uE$ from its values on half-spaces
yields its \emph{maximal extension}
\begin{equation}
  \label{eq:14}
  \ouE(X)=\bigcap_{\eta\in\Lp[0](\Sphere\cap\GG)} \uE(\HX),
\end{equation}
being the intersection of superlinear expectations of random
half-spaces $\HX=H_\eta(h(X,\eta))$ almost surely containing
$X\in\Lp(\sFC)$. Recall that $\GG=\cone^o$. 

\begin{proposition}
  \label{pro:ext-super}
  If $\uE$ is superlinear on half-spaces with the same normal, that is,
  \begin{equation}
    \label{eq:18}
    \uE(H_\eta(\beta+\beta'))\supset \uE(H_\eta(\beta))+\uE(H_\eta(\beta'))
  \end{equation}
  for $\beta,\beta'\in\Lp(\R)$ and $\eta\in\Lp[0](\Sphere\cap\GG)$, and is
  scalarly upper semicontinuous on half-spaces with the same normal,
  that is,
  \begin{displaymath}
    \uE(H_\eta(\beta))\supset \limsup \uE(H_\eta(\beta_n))
  \end{displaymath}
  if $\beta_n\to\beta$ in $\sigma(\Lp,\Lp[q])$, then its maximal
  extension $\ouE(X)$ given by \eqref{eq:14} is superlinear and upper
  semicontinuous with respect to the
  scalar convergence of random closed convex sets. If $\uE$ is law
  invariant on half-spaces, then $\ouE$ is law invariant. 
\end{proposition}
\begin{proof}
  The additivity on deterministic singletons follows from the fact
  that $\HX[X+a]=\HX+a$ for all $a\in\R^d$. If $F\in\sFC$ is
  deterministic, then 
  \begin{displaymath}
    \ouE(F) \subset \bigcap_{u\in\Sphere\cap\GG} \uE(\HuX[F])
    \subset \bigcap_{u\in\Sphere\GG} \HuX[F]=F.
  \end{displaymath}
  The homogeneity and monotonicity properties of the extension are
  obvious.  For two $p$-integrable random closed convex sets $X$ and
  $Y$, \eqref{eq:18} yields that
  \begin{align*}
    \ouE(X+Y)\;&=\; \bigcap_{\eta\in\Lp[0](\Sphere\cap\GG)} 
    \uE(H_\eta(h(X,\eta)+h(Y,\eta)))\\
    &\supset\; \bigcap_{\eta\in\Lp[0](\Sphere\cap\GG)} 
    \uE(\HX)+\uE(\HX[Y])\\
    &\supset\; \bigcap_{\eta\in\Lp[0](\Sphere\cap\GG)} 
    \uE(\HX)+ \bigcap_{\eta\in\Lp[0](\Sphere\cap\GG)}\uE(\HX[Y])\\
    &=\; \ouE(X)+\ouE(Y).
  \end{align*}
  Assume that $X_n$ scalarly converges to $X$. Let
  $x_{n_k}\in\uE(X_{n_k})$ and $x_{n_k}\to x$ for some $x$. Then
  $x_{n_k}\in\uE(H_\eta(X_{n_k}))$ for all
  $\eta\in\Lp[0](\Sphere\cap\GG)$. Since $h(X_{n_k},\eta)\to h(X,\eta)$ in
  $\sigma(\Lp,\Lp[q])$, upper semicontinuity on half-spaces yields
  that $\uE(\HX)\supset \limsup \uE(H_\eta(X_{n_k}))$, whence
  $x\in\uE(\HX)$ for all $\eta$. Therefore, $x\in\ouE(X)$, confirming
  the upper semicontinuity of the maximal extension. The law
  invariance property is straightforward. 
\end{proof}

It is possible to let $\eta$ in \eqref{eq:14} be deterministic and define
\begin{equation}
  \label{eq:9}
  \ouEnot(X)=\bigcap_{u\in\Sphere\cap\GG} \uE(\HuX).
\end{equation}
With this \emph{reduced maximal extension}, the superlinear
expectation is extended from its values on half-spaces with
deterministic normal vectors. 
Note that the reduced maximal extension may be equal to the whole
space, e.g., for $X=H_\eta(0)$ being a half-space with a
nondeterministic normal. It is obvious that $\uE(X)\subset
\ouE(X)\subset \ouEnot(X)$ and $\ouEnot$ is constant preserving. 
The reduced maximal extension is particularly useful for Hausdorff
approximable random closed sets. 

\subsection{Exact nonlinear expectations}
\label{sec:exact-nonl-expect}

It is possible to apply the maximal extension to the sublinear
expectation and the minimal extension to the superlinear one,
resulting in $\ovE$ and $\uuE$.  The monotonicity property yields
that, for each $p$-integrable random closed set $X$,
\begin{equation}
  \label{eq:10}
  \uvE(X) \subset \vE(X) \subset \ovE(X)\subset \ovEnot(X).
\end{equation}
It is easy to see that each extension is an idempotent operation,
e.g., the minimal extension of $\uvE$ coincides with $\uvE$.

A nonlinear sublinear expectation is said to be \emph{minimal}
(respectively, maximal) if it coincides with its minimal
(respectively, maximal) extension. The superlinear expectation is said
to be \emph{reduced maximal} if $\uE$ coincides with $\ouEnot$.
Since random convex closed sets can be represented either as families
of their selections or as intersections of half-spaces, the minimal
representation may be considered a primal representation of an exact
nonlinear expectation, while the maximal representation becomes the
dual one. 

If \eqref{eq:10} holds with the equalities, then $\vE$ is said to be
\emph{exact}. The same applies to superlinear expectations.  Note that
the selection expectation is exact on all integrable random closed
convex sets, its minimality corresponds to \eqref{eq:3} and maximality
becomes \eqref{eq:16}.

\section{Sublinear set-valued expectations}
\label{sec:subl-expect}

\subsection{Duality for minimal sublinear expectations}
\label{sec:minim-subl-expect}

The minimal sublinear expectation is determined by its restriction on
random sets $\xi+\cone$; the following result characterises such a
restriction. 

\begin{lemma}
  \label{thr:dsub}
  A map $(\xi+\cone)\mapsto \vE(\xi+\cone)\in \co\sF$ for
  $\xi\in\Lp(\R^d)$ is a $\sigma(\Lp,\Lp[q])$-lower semicontinuous
  normalised sublinear expectation if and only if
  $h(\vE(\xi+\cone),u)=\infty$ for $u\notin\GG=\cone^o$, and
  \begin{equation}
    \label{eq:17}
    h(\vE(\xi+\cone),u)
    =\sup_{\zeta\in\sZ_u, \E\zeta=u}\E\langle\zeta,\xi\rangle,\quad
    u\in\GG, 
  \end{equation}
  where $\sZ_u$, $u\in\GG$, are convex $\sigma(\Lp[q],\Lp)$-closed
  cones in $\Lp[q](\GG)$, such that $\{\E\zeta:\;
  \zeta\in\sZ_u\}=\{tu:\; t\geq0\}$ for all $u\neq0$, $\sZ_{cu}=\sZ_u$
  for all $c>0$, $\sZ_0=\{0\}$, and
  \begin{equation}
    \label{eq:21}
    \sZ_{u+v}\subset \sZ_u+\sZ_v,\quad u,v\in\GG.
  \end{equation}
\end{lemma}
\begin{proof}
  \textsl{Sufficiency.} For linearly independent $u$ and $v$, each
  $\zeta\in\sZ_{u+v}$ satisfies $\zeta=\zeta_1+\zeta_2$ with
  $\E\zeta_1=t_1u$ and $\E\zeta_2=t_2v$. Thus, $\E\zeta=t(u+v)$ only
  if $t_1=t_2=t$. Therefore,
  \begin{align*}
    h(\vE(\xi+\cone),u+v)
    &=\sup_{\zeta\in\sZ_{u+v},\E\zeta=u+v}\;\E\langle\zeta,\xi\rangle\\
    &\leq \sup_{\zeta\in\sZ_u+\sZ_v
      ,\E\zeta=u+v}\E\langle\zeta,\xi\rangle\\
    &\leq
    \sup_{\zeta_1\in\sZ_u,\zeta_2\in\sZ_v,\E\zeta_1=u,\E\zeta_2=v}
    \E\langle\zeta_1+\zeta_2,\xi\rangle\\
    &\leq h(\vE(\xi+\cone),u)+h(\vE(\xi+\cone),v).
  \end{align*}
  Since $\sZ_{cu}=\sZ_u=c\sZ_u$ for any $c>0$,
  \begin{align*}
    h(\vE(\xi+\cone),cu)
    &=\sup_{\zeta\in\sZ_{cu}, \E\zeta=cu}\E\langle\zeta,\xi\rangle
    =\sup_{\zeta'\in\sZ_u, \E\zeta'=u}\E\langle c\zeta',\xi\rangle\\
    &=ch(\vE(\xi+\cone),u),
  \end{align*}
  whence the function $h(\vE(\xi+\cone),u)$ is sublinear in $u$ and so
  is a support function. 

  The additivity property on singletons follows from the
  construction, since 
  \begin{displaymath}
    \sup_{\zeta\in\sZ_u, \E\zeta=u}\E\langle\zeta,\xi+a\rangle
    =\sup_{\zeta\in\sZ_u, \E\zeta=u}\E\langle\zeta,\xi\rangle+\langle
    a,u\rangle 
  \end{displaymath}
  for each deterministic $a\in\R^d$. Furthermore, 
  $h(\vE(\cone),u)=h(\cone,u)$, whence $\vE(\cone)=\cone$. 
  The homogeneity property is obvious.  The function $\vE$ is
  subadditive, since
  \begin{displaymath}
    h(\vE(\xi+\eta+\cone),u)=\sup_{\zeta\in\sZ_u,\E\zeta=u}\langle
    u,\xi+\eta\rangle
    \leq h(\vE(\xi+\cone),u)+h(\vE(\eta+\cone),u).
  \end{displaymath}

  For $u\in\GG$, the set $\{\zeta\in\sZ_u:\; \E\zeta=u\}$ is closed in
  $\sigma(\Lp[q],\Lp)$. Indeed, if $\zeta_n\to\zeta$, then in
  $\E\langle \zeta_n,\xi\rangle\to\E\langle\zeta,\xi\rangle$ let $\xi$
  be one the basis vectors to confirm that $\E\zeta=u$. Since
  $h(\vE(\xi+\cone),u)$ is the support function of the closed set
  $\{\zeta\in\sZ_u:\; \E\zeta=u\}$ in direction $\xi$, it is lower
  semicontinuous as function of $\xi\in\Lp(\R^d)$, so that
  \eqref{eq:2} holds.  
  
  \smallskip
  \noindent
  \textsl{Necessity.}  By Proposition~\ref{prop:GG}, the support
  function is infinite for $u\notin\GG$. For $u\in\GG$, let $\sA_u$ be
  the set of $\xi\in \Lp(\R^d)$ such that $h(\vE(\xi+\cone),u)\leq
  0$. The map $\xi\mapsto h(\vE(\xi+\cone),u)$ is a sublinear map from
  $\Lp(\R^d)$ to $(-\infty,\infty]$. By sublinearity, $\sA_u$ is a
  convex cone in $\Lp(\R^d)$, and $\sA_{cu}=\sA_u$ for all
  $c>0$. Furthermore, $\sA_u$ is closed with respect to the scalar
  convergence $\xi_n+\cone\to \xi+\cone$ by the assumed lower
  semicontinuity of $\vE$. Hence, it is closed with respect to the
  convergence $\xi_n\to\xi$ in $\sigma(\Lp,\Lp[q])$.  

  Note that $0\in\sA_u$, and let
  \begin{displaymath}
    \sZ_u=\{\zeta\in \Lp[q](\R^d):\; \E\langle\zeta,\xi\rangle\leq 0\;
    \text{for all} \; \xi\in\sA_u\}
  \end{displaymath}
  be the polar cone to $\sA_u$. For $u=0$, we have $\sA_0=\Lp(\R^d)$
  and $\sZ_0=\{0\}$. 
  Consider $u\neq0$.  Letting $\xi=a\one_H$ for an event $H$ and
  deterministic $a$ with $\langle a,u\rangle\leq 0$, we obtain a
  member of $\sA_u$, whence each $\zeta\in\sZ_u$ satisfies
  $\langle\E\zeta,a\one_H\rangle\leq 0$ whenever $\langle
  a,u\rangle\leq 0$. Thus, $\zeta\in G$ a.s., and letting $H=\Omega$
  yields that $\E\zeta=tu$ for some $t\geq0$ and all $\zeta\in\sZ_u$.
  The subadditivity property of the support function of
  $\vE(\xi+\cone)$ yields that $\sA_{u+v}\supset (\sA_u\cap\sA_v)$ for
  $u,v\in\GG$. By a Banach space analogue of \cite[Th.~1.6.9]{schn2},
  the polar to $\sA_u\cap\sA_v$ is the closed sum $\sZ_u+\sZ_v$ of the
  polars, whence \eqref{eq:21} holds.

  By the definition of $\sA_u$,
  \begin{displaymath}
    h(\vE(\xi+\cone),u)=\inf\big\{\langle x,u\rangle:\; \xi-x\in\sA_u\big\}.
  \end{displaymath}
  Since $\sA_u$ is convex and $\sigma(\Lp,\Lp[q])$-closed,
  the bipolar theorem yields that
  \begin{align*}
    h(\vE(\xi+\cone),u)&=\inf\{\langle x,u\rangle:\; \xi-x\in\sA_u\}\\
    &=\inf\big\{\langle x,u\rangle:\; \E\langle
    \zeta,\xi-x\rangle \leq 0\;\text{for all}\; \zeta\in\sZ_u\big\}\\
    &=\sup_{\zeta\in\sZ_u,\E\zeta=u}\E\langle\zeta,\xi\rangle. \qedhere
  \end{align*}
\end{proof}

\begin{theorem}
  \label{thr:sub-sets-dual}
  A function $\vE$ on $p$-integrable random closed convex sets is a
  scalarly lower semicontinuous minimal normalised sublinear
  expectation if and only if $\vE$ admits the representation
  \begin{equation}
    \label{eq:24}
    h(\vE(X),u) =\sup_{\zeta\in\sZ_u,\E\zeta=u} \E h(X,\zeta),\quad u\in\GG,
  \end{equation}
  and $h(\vE(X),u)=\infty$ for $u\notin\GG$, 
  where $\{\sZ_u,u\in\R^d\}$ satisfy the conditions of
  Lemma~\ref{thr:dsub}.
\end{theorem}
\begin{proof}
  \textsl{Necessity.}  Lemma~\ref{thr:dsub} applies to the restriction
  of $\vE$ onto random sets $\xi+\cone$.  By the minimality
  assumption, $\vE$ coincides with its minimal extension 
  \eqref{eq:23}. By Lemma~\ref{thr:dsub}, for $u\in\GG$,
  \begin{align*}
    h(\vE(X),u)&=\sup_{\xi\in\Lp(X)}\; \sup_{\zeta\in\sZ_u,
      \E\zeta=u}\E\langle\zeta,\xi\rangle
    =\sup_{\zeta\in\sZ_u,\E\zeta=u}
    \E\sup_{\xi\in\Lp(X)}\langle\zeta,\xi\rangle\\
    &=\sup_{\zeta\in\sZ_u,\E\zeta=u} \E h(X,\zeta),
  \end{align*}
  where \eqref{eq:20} has been used. 

  \smallskip
  \noindent
  \textsl{Sufficiency.} The right-hand side of \eqref{eq:24} is
  sublinear in $u$ and so is a support function. The additivity on
  singletons, monotonicity, subadditivity and homogeneity properties
  of $\vE$ are obvious. For a deterministic $F\in\sFC$, the
  sublinearity of the support function yields that
  \begin{displaymath}
    h(\vE(F),u)=\sup_{\zeta\in\sZ_u,\E\zeta=u} \E h(F,\zeta)
    \geq \sup_{\zeta\in\sZ_u,\E\zeta=u} h(F,\E\zeta)=h(F,u),
  \end{displaymath}
  whence $\vE(F)\supset F$.

  The minimality of $\vE$ follows from 
  \begin{displaymath}
    \uvE(X)=
    \sup_{\xi\in\Lp(X)} \sup_{\zeta\in\sZ_u,\E\zeta=u} \E \langle\xi,\zeta\rangle
    =\sup_{\zeta\in\sZ_u,\E\zeta=u} \E h(X,\zeta)=\vE(X).
  \end{displaymath}
  Since the support function of $\vE(X)$ given by \eqref{eq:24} is the
  supremum of scalarly continuous functions of $X$, the minimal
  sublinear expectation is scalarly lower semicontinuous.
\end{proof}

\begin{corollary}
  \label{cor:sand-E}
  If $u\in\sZ_u$ for all $u\in\R^d$, then $\E X\subset \vE(X)$ for all
  $p$-integrable $X$ and any scalarly lower semicontinuous normalised
  minimal sublinear expectation $\vE$.
\end{corollary}
\begin{proof}
  By \eqref{eq:24}, $h(\vE(X),u)\leq \E h(X,u)=h(\E X,u)$ for all
  $u\in\GG$. 
\end{proof}

\begin{remark}
  \label{rem:law-inva}
  The sublinear expectation given by \eqref{eq:24} is law invariant if
  and only if the sets $\sZ_u$ are \emph{law-complete}, that is, with
  each $\zeta\in\sZ_u$, the set $\sZ_u$ contains all random
  vectors that share distribution with $\zeta$.
\end{remark}

\begin{example}
  \label{ex:matrix}
  Let $Z$ be a random matrix with $\E Z$ being the identity matrix,
  and let $\sZ_u=\{tZu^\top:\; t\geq0\}$, $u\in\GG=\R^d$. Then
  \eqref{eq:24} turns into $h(\vE(X),u)=\E h(Z^\top X,u)$, whence
  $\vE(X)=\E(Z^\top X)$. It is possible to let $Z$ belong to a family
  of such matrices; then $\vE(X)$ is the closed convex hull of the
  union of $\E(Z^\top X)$ for all such $Z$. In this example,
  $h(\vE(X),u)$ is not solely determined by $h(X,u)$. This sublinear
  expectation is not necessarily constant preserving. 
\end{example}

\begin{example}[Random half-space]
  \label{ex:hs}
  Let $X=H_\eta(\beta)$ with $\beta\in\Lp(\R)$ and
  $\eta\in\Lp[0](\Sphere\cap\GG)$. By \eqref{eq:24}, $h(\vE(X),u)$ is
  finite for $u\in\Sphere\cap\GG$ only if each $\zeta\in\sZ_u$ with
  $\E\zeta=u$ satisfies $\zeta=\gamma\eta$ a.s. with
  $\gamma\in\Lp[q](\R_+)$. Then
  \begin{displaymath}
    h(\vE(H_\eta(\beta)),u)
    =\sup_{\gamma\in\Lp[q](\R_+),\gamma\eta\in\sZ_u,\E(\gamma\eta)=u}
    \E(\gamma\beta). 
  \end{displaymath}
  If the normal $\eta=u$ is deterministic and 
  \begin{equation}
    \label{eq:40}
    \sZ_u\subset \{\gamma u:\; \gamma\in\Lp[q](\R_+)\},
  \end{equation}
  then $\vE(H_u(\beta))=H_u(t)$ with
  \begin{displaymath}
    t=\sup_{\gamma u\in\sZ_u,\E\gamma=1} \E(\gamma\beta).
  \end{displaymath}
  Otherwise, $\vE(H_u(\beta))=\R^d$. 
\end{example}

\subsection{Exact sublinear expectation}
\label{sec:supp-funct-appr}

Consider now the situation when, for each $u$, the value of 
$h(\vE(X),u)$ is solely determined by
the distribution of $h(X,u)$. This is the case if the
supremum in \eqref{eq:24} involves only $\zeta$ such that
$\zeta=\gamma u$ for some $\gamma\in\Lp[q](\R_+)$. The following
result shows that this condition characterises constant preserving
minimal sublinear expectations, which then necessarily become exact
ones. 

\begin{theorem}
  \label{thr:exact}
  A function $\vE$ on $p$-integrable random closed convex sets from
  $\Lp(\sFC)$ is a scalarly lower semicontinuous constant preserving
  minimal sublinear expectation if and only if $h(\vE(X),u)=\infty$
  for $u\notin\GG$, and
  \begin{equation}
    \label{eq:24a}
    h(\vE(X),u) =\sup_{\gamma\in\sM_u,\E\gamma=1} \E (\gamma h(X,u)),\quad u\in\GG,
  \end{equation}
  where $\sM_u$, $u\in\GG$, are convex $\sigma(\Lp[q],\Lp)$-closed
  cones in $\Lp[q](\R_+)$, such that $\sM_{cu}=\sM_u$ for all $c>0$,
  and $\sM_{u+v}\subset \sM_u\cap \sM_v$ for all $u,v\in\R^d$.
\end{theorem}
\begin{proof}
  \textsl{Sufficiency.}  If $\sM_u$, $u\in\R^d$, satisfy the imposed
  conditions, then $\sZ_u=\{\gamma u:\; \gamma\in\sM_u\}$, $u\in\GG$,
  satisfy the conditions of Lemma~\ref{thr:dsub}. Indeed,
  $\sZ_{cu}=\sZ_u$ for all $c>0$, and 
  \begin{displaymath}
    \sZ_{u+v}=\{\gamma (u+v):\; \gamma\in\sM_{u+v}\}
    \subset \{\gamma (u+v):\; \gamma\in\sM_u\cap\sM_v\}
    \subset \sZ_u+\sZ_v
  \end{displaymath}
  for all $u,v\in\GG$.  If $F\in\sFC$ is deterministic, then
  \begin{displaymath}
    h(\vE(F),u) =\sup_{\gamma\in\sM_u,\E\gamma=1} \E h(F,\gamma u)
    =h(F,u), \quad u\in\GG,
  \end{displaymath}
  whence $\vE$ is constant preserving. 

  \smallskip
  \noindent
  \textsl{Necessity.} Since $\vE$ is minimal, 
  the support function of $\vE(X)$ is given by \eqref{eq:24}. The
  constant preserving property yields that $\vE(H_u(t))=H_u(t)$ for
  all half-spaces $H_u(t)$ with $u\in\GG$. 
  By the argument from Example~\ref{ex:hs},
  the minimal sublinear expectation of a half-space $H_u(t)$ is
  distinct from the whole space only if \eqref{eq:40} holds. 

  The properties of $\sZ_u$ imply the imposed properties of
  $\sM_u=\{\gamma:\; \gamma u\in\sZ_u\}$.  Indeed, assume that
  $\gamma\in\sM_{u+v}$, so that $\gamma(u+v)\in\sZ_{u+v}$. Hence,
  $\gamma(u+v)\in(\sZ_u+\sZ_v)$, meaning that $\gamma(u+v)$ is the
  norm limit of $\gamma_{1n}u+\gamma_{2n}v$ for $\gamma_{1n}u\in\sZ_u$
  and $\gamma_{2n}v\in\sZ_v$, $n\geq1$. The linear independence of $u$
  and $v$ yields that $\gamma_{1n}\to\gamma$ and
  $\gamma_{2n}\to\gamma$, whence $\gamma\in(\sM_u\cap\sM_v)$.
\end{proof}

It is possible to rephrase \eqref{eq:24a} as 
\begin{equation}
  \label{eq:32}
  h(\vE(X),u)=\ve_u(h(X,u)), \quad u\in\GG,
\end{equation}
for numerical sublinear expectations 
\begin{displaymath}
  \ve_u(\beta) =\sup_{\gamma\in\sM_u,\E\gamma=1}\E(\gamma\beta),
  \quad u\in\GG, \;\;\beta\in\Lp(\R),
\end{displaymath}
defined by an analogue of \eqref{eq:supremum-vE}. 
Since the negative part of $h(X,u)$ is $p$-integrable, it is
possible to consistently let $\ve(h(X,u))=\infty$ in \eqref{eq:4}
if $h(X,u)$ is not $p$-integrable.

\begin{corollary}
  \label{cor:sub-exact}
  Each scalarly lower semicontinuous constant preserving minimal
  sublinear expectation is exact.
\end{corollary}
\begin{proof}
  Since \eqref{eq:24a} yields that $\vE(\HX)=\R^d$ if $\eta$ is
  random, the maximal extension of $\vE$ by an analogue of
  \eqref{eq:14} reduces to deterministic $\eta$ and so $\ovE=\ovEnot$
  is the reduced maximal extension.  For $u\in\Sphere\cap\GG$ and
  $\beta\in\Lp(\R)$, we have $\vE(H_u(\beta))=H_u(\ve_u(\beta))$,
  cf. Example~\ref{ex:hs}.  Thus, the reduced maximal extension of
  $\vE$ is given by
  \begin{displaymath}
    \ovEnot(X)=\bigcap_{u\in\Sphere\cap\GG} H_u(\ve_u(h(X,u))). 
  \end{displaymath}
  Comparing with \eqref{eq:32}, we see that 
  $\ovEnot(X)\subset \vE(X)$. The opposite inclusion is obvious,
  whence $\ovEnot(X)=\ovE(X)=\vE(X)$. 
\end{proof}

\begin{corollary}
  \label{cor:max-plus}
  If $\vE$ is a scalarly lower semicontinuous constant preserving
  minimal normalised sublinear expectation, then $\vE(X+F)=\vE(X)+F$ for each
  deterministic $F\in\sFC$.
\end{corollary}

\begin{corollary}
  \label{cor:sand-E-cp}
  Assume that $\vE$ is scalarly lower semicontinuous constant
  preserving minimal law invariant sublinear expectation.
  Then $\vE(\E(X|\ssalg))\subset\vE(X)$ for all $X\in\Lp(\sFC)$ and
  any sub-$\sigma$-algebra $\ssalg$ of $\salg$. In particular, $\E
  X\subset \vE(X)$.
\end{corollary}
\begin{proof}
  The law invariance of $\vE$ implies that $\ve_u$ is law invariant.
  The sublinear expectation $\ve_u$ is \emph{dilatation monotonic},
  meaning that $\ve_u(\E(\beta|\ssalg))\leq \ve_u(\beta)$ for all
  $\beta\in\Lp(\R)$, see \cite[Cor.~4.59]{foel:sch04} for this fact
  derived for risk measures. The statement follows from
  \eqref{eq:32}. 
\end{proof}

For a $p$-integrable random closed convex set $X$, its Firey
$p$-expectation is defined by $h(\E_p X,u)=(\E h(X,u)^p)^{1/p}$. The
next result follows from H\"older's inequality applied to $\E(\gamma
h(X,u))$ in \eqref{eq:24a}. 

\begin{corollary}
  \label{cor:upper-bound-vE}
  If $\vE$ admits representation \eqref{eq:24a}, then 
  \begin{displaymath}
    \vE(X)\subset (\E_p X)\sup_{u\in\GG,\gamma\in\sM_u,\E\gamma=1}
    (\E\gamma^q)^{1/q}.
  \end{displaymath}
\end{corollary}

The following result identifies a particularly important case, when
the families $\sM_u=\sM$ do not depend on $u$. This property
essentially means that the sublinear expectation preserves centred
balls. By $B_r$ denote the ball of radius $r$ centred at the origin. 

\begin{theorem}
  \label{thr:sub-dual}
  A scalarly lower semicontinuous constant preserving minimal
  superlinear expectation $\vE$ satisfies
  $\vE(B_\beta+\cone)=B_r+\cone$ for all $\beta\in\Lp(\R_+)$ and
  $r\geq 0$ (depending on $\beta$) if and only if \eqref{eq:24a} holds
  with $\sM_u=\sM$ for all $u\neq0$. Then
  \begin{equation}
    \label{eq:4}
    h(\vE(X),u)=\ve(h(X,u)),\quad u\in\GG,
  \end{equation}
  where $\ve$ admits the representation \eqref{eq:supremum-vE}.
  Furthermore,
  \begin{equation}
    \label{eq:subX-union}
    \vE(X)=\cco\bigcup_{\gamma\in\sM,\E\gamma=1} \E(\gamma X).
  \end{equation}
\end{theorem}
\begin{proof}
  Assume that $\sM_u$ are constructed as in the proof of
  Theorem~\ref{thr:exact}, so that $\sM_u$ is maximal for each
  $u\in\GG$.  The right-hand side of 
  \begin{displaymath}
    h(\vE(B_\beta+\cone),u) =\sup_{\gamma\in\sM_u,\E\gamma=1} \E (\gamma \beta).
  \end{displaymath}
  does not depend on $u\in\Sphere\cap\GG$ if and only if $\sM_u=\sM$
  for all $u\in\GG$.

  Representation \eqref{eq:4} follows from \eqref{eq:32} with
  $\sM_u=\sM$. In view of \eqref{eq:supremum-vE},
  \begin{align*}
    \sup_{\gamma\in\sM,\E\gamma=1} \E h(\gamma X,u)
    &=\sup_{\gamma\in\sM,\E\gamma=1} \E h(\gamma X,u)
    =\sup_{\gamma\in\sM,\E\gamma=1} \E(\gamma h(X,u))\\ &=\ve(h(X,u))\,. 
  \end{align*}
  By \eqref{eq:4}, the support functions of the both sides of
  \eqref{eq:subX-union} are identical.
\end{proof}  

If $X=\{\xi\}$ is a singleton, there is no need to take the convex
hull on the right-hand side of \eqref{eq:subX-union}.

\begin{example}
  \label{ex:chull}
  For an integrable $X$ and $n\geq1$, consider the sublinear
  expectation
  \begin{displaymath}
    \vE_n^\cup(X)=\E \co(X_1\cup\cdots\cup X_n),
  \end{displaymath}
  It is easy to see that $\vE_n^\cup(X)$ is a minimal constant
  preserving sublinear expectation; it is given by \eqref{eq:4} with
  the corresponding numerical sublinear expectation $\ve(\beta)$,
  being the expected maximum of $n$ i.i.d.\ copies of
  $\beta\in\Lp[1](\R)$. By Corollary~\ref{cor:sub-exact}, this
  sublinear expectation is exact.
\end{example}

\begin{example}
  \label{ex:avar}
  For $\alpha\in(0,1)$, let $\sP_\alpha$ be the family of random
  variables $\gamma$ with values in $[0,\alpha^{-1}]$ and such that
  $\E\gamma=1$. Furthermore, let $\sM$ be the cone generated by
  $\sP_\alpha$, that is $\sM=\{t\gamma:\gamma\in\sP_\alpha, t\geq0\}$. In
  finance, the set $\sP_\alpha$ generates the average Value-at-Risk,
  which is the risk measure obtained as the average quantile, see
  \cite{foel:sch04}. Similarly, the numerical sublinear $\ve$ and
  superlinear $\ue$ expectations generated by this set $\sM$ are
  represented as average quantiles. Namely, $\ve(\beta)$ is the
  average of the quantiles of $\beta$ at levels $t\in(1-\alpha,1)$,
  and $\ue(\beta)$ is the average of the quantiles at levels
  $t\in(0,\alpha)$. The corresponding set-valued sublinear expectation
  $\vE$ satisfies $\E X\subset \vE(X)\subset \alpha^{-1}\E X$. 
\end{example}

\section{Superlinear set-valued expectations}
\label{sec:superl-expect-1}

\subsection{Duality for maximal superlinear expectations}
\label{sec:maxim-superl-expect-1}

Consider a superlinear expectation defined on $\Lp(\sFC)$. If
$\cone=\{0\}$, we deal with all $p$-integrable random closed convex
sets. Recall that $\GG=\cone^o$ is the polar cone to $\cone$.

\begin{theorem}
  \label{thr:nha}
  A map $\uE:\Lp(\sFC)\mapsto\co\sF$ is a scalarly upper
  semicontinuous normalised maximal superlinear expectation if and
  only if
  \begin{equation}
    \label{eq:27}
    \uE(X)=\bigcap_{\eta\in\Lp[0](\Sphere\cap\GG)}\; \bigcap_{\gamma\in\sM_\eta}
    \big\{x:\; \langle x,\E(\gamma\eta)\rangle\leq \E h(X,\gamma\eta)\big\}
  \end{equation}
  for a collection of convex $\sigma(\Lp[q],\Lp)$-closed cones
  $\sM_\eta\subset\Lp[q](\R_+)$ parametrised by
  $\eta\in\Lp[0](\Sphere\cap\GG)$ and such that $\sM_u$ is strictly
  larger than $\{0\}$ 
  for each deterministic $\eta=u\in\Sphere\cap\GG$. 
\end{theorem}
\begin{proof}
  \textsl{Necessity.} Fix $\eta\in\Lp[0](\Sphere\cap\GG)$, and let
  $\sA_\eta$ be the set of $\beta\in\Lp(\R)$ such that
  $\uE(H_\eta(\beta))$ contains the origin. Since
  $\uE(H_\eta(0))\supset \uE(\cone)=\cone$, we have
  $0\in\sA_\eta$. Since $\uE(H_u(t))\subset H_u(t)$, the family
  $\sA_u$ does not contain $\beta=t$ for $t<0$ and
  $u\in\Sphere\cap\GG$. 

  If $\beta_n\to\beta$ in $\sigma(\Lp,\Lp[q])$, then $\E
  h(H_\eta(\beta_n),\gamma\eta)\to \E h(H_\eta(\beta),\gamma\eta)$ for all
  $\gamma\in\Lp[q](\R)$, whence $H_\eta(\beta_n)\to H_\eta(\beta)$
  scalarly in $\sigma(\Lp,\Lp[q])$. Therefore,
  \begin{displaymath}
    \uE(H_\eta(\beta))\supset\limsup\; \uE(H_\eta(\beta_n))
  \end{displaymath}
  by the assumed upper semicontinuity of $\uE$. Thus, $\sA_\eta$ is a
  convex $\sigma(\Lp,\Lp[q])$-closed cone in $\Lp(\R)$. Consider its
  positive dual cone
  \begin{displaymath}
    \sM_\eta=\big\{\gamma\in\Lp[q](\R):\; \E(\gamma\beta)\geq0\; \text{for
      all}\; \beta\in\sA_\eta\big\}.
  \end{displaymath}
  Since $\uE(\cone)=\cone$, we have $\uE(X)\ni 0$ whenever
  $\cone\subset X$ a.s. In view of this, if $\beta$ is
  a.s. nonnegative, then $H_\eta(\beta)$ a.s. contains zero and so
  $\beta\in\sA_\eta$. Thus, each $\gamma$ from $\sM_\eta$ is
  a.s. nonnegative.  The bipolar theorem yields that
  \begin{equation}
    \label{eq:26}
    \sA_\eta=\big\{\beta\in\Lp(\R):\; \E(\gamma\beta)\geq 0\; \text{for all}\;
    \gamma\in\sM_\eta\big\}. 
  \end{equation}
  Since $(-t)\notin\sA_u$, \eqref{eq:26} yields that the cone $\sM_u$
  is strictly larger than $\{0\}$.  Since $\uE$ is assumed to be
  maximal, \eqref{eq:14} implies that
  \begin{align*}
    \uE(X)=\; \ouE(X)&=\bigcap_{\eta\in\Lp[0](\Sphere\cap\GG)}
    \big\{x:\; \uE(\HX-x)\ni 0\big\}\\
    &=\bigcap_{\eta\in\Lp[0](\Sphere\cap\GG)}
    \big\{x:\; h(X,\eta)-\langle x,\eta\rangle\in \sA_\eta\big\}\\
    &=\bigcap_{\eta\in\Lp[0](\Sphere\cap\GG)}
    \bigcap_{\gamma\in\sM_\eta}
    \big\{x:\; \E \langle x,\gamma\eta\rangle\leq \E h(X,\gamma\eta)\}.
  \end{align*}
  
  \noindent
  \textsl{Sufficiency.} It is easy to check that $\uE$ given by
  \eqref{eq:27} is additive on deterministic singletons,
  homogeneous and monotonic. If $F\in\sFC$ is
  deterministic, then letting $\eta=u$ in \eqref{eq:27} be
  deterministic and using the nontriviality of $\sM_u$ yield that
  $\uE(F)\subset F$. Furthermore, $\uE(\cone)=\cone$, since
  $\uE(\cone)$ contains the origin and so is not empty.

  The superadditivity of $\uE$ follows from the fact that
  \begin{multline*}
    \big\{x: \langle x,\E(\gamma\eta)\rangle\leq \E h(X,\gamma\eta)+\E
    h(Y,\gamma\eta)\big\}\\
    \supset \big\{x: \langle x,\E(\gamma\eta)\rangle\leq \E
    h(X,\gamma\eta)\big\}
    +\big\{x: \langle x,\E(\gamma\eta)\rangle\leq \E
    h(Y,\gamma\eta)\big\}.
  \end{multline*}
  It is easy to see that $\uE$ coincides with its maximal extension.

  Note that \eqref{eq:27} is equivalently written as 
  \begin{displaymath}
    \uE(X)=\bigcap_{\eta\in\Lp[0](\Sphere\cap\GG)}\; \bigcap_{\gamma\in\sM_\eta}
    \big\{x:\; \E h(X-x,\gamma\eta)\geq 0\big\}.
  \end{displaymath}
  If $X_n$ scalarly converges to $X$ and $x_{n_k}\to x$ for
  $x_{n_k}\in\uE(X_{n_k})$, $k\geq1$, then $\E h(X_n-x_n,\gamma\eta)$
  converges to $\E h(X-x,\gamma\eta)$ for all 
  $\gamma\in\Lp[q](\R_+)$ and $\eta\in\Lp[0](\Sphere\cap\GG)$. Thus, 
  $\E h(X-x,\gamma\eta)\geq 0$,
  whence $x\in\uE(X)$, and the upper semicontinuity of $\uE$ follows.
\end{proof}

In difference to the sublinear case (see
Theorem~\ref{thr:sub-sets-dual}), the cones $\sM_\eta$ from
Theorem~\ref{thr:nha} do not need to satisfy additional conditions
like those imposed in Lemma~\ref{thr:dsub}.

\begin{corollary}
  \label{cor:sand-U}
  If $1\in\sM_\eta$ for all $\eta$, then $\uE(X)\subset \E X$ for all
  $p$-integrable $X$ and any scalarly upper semicontinuous maximal
  normalised superlinear expectation $\uE$.
\end{corollary}
\begin{proof}
  Restrict the intersection in \eqref{eq:27} to deterministic
  $\eta=u$ and $\gamma=1$, so that the right-hand side of \eqref{eq:27}
  becomes $\E X$. 
\end{proof}

\begin{example}
  \label{ex:half-space-super}
  Let $X=H_\eta(\beta)$ be the half-space with normal
  $\eta\in\Lp[0](\Sphere)$ and $\beta\in\Lp(\R)$. If $\cone=\{0\}$,
  the maximal superlinear expectation of $X$ is given by
  \begin{displaymath}
    \uE(H_\eta(\beta)) =\bigcap_{\gamma\in\sM_\eta}
    \big\{x:\; \langle x,\E(\gamma\eta)\rangle\leq \E(\gamma\beta)\big\}.
  \end{displaymath}
  Assume that $d=2$ and let $\eta=(1,\pi)$ with $\pi$ being an almost
  surely positive random variable.  We have
  \begin{align*}
    \uE(H_\eta(\beta))&=\bigcap_{\gamma\in\sM_\eta,\E\gamma=1}
    \big\{(x_1,x_2):\; x_1+x_2\E(\gamma\pi)\leq
    \E(\gamma\beta)\big\}\\
    &=\big\{(x_1,x_2):\; x_1\leq \ue(\beta-x_2\pi)\big\},
  \end{align*}
  where $\ue$ is the numerical superlinear expectation with the
  generating set $\sM_\eta$.  In particular, if $\beta=0$ a.s., then
  \begin{multline*}
    \uE(H_\eta(0))=
    \big\{(x_1,x_2):\; x_2\geq0, x_1\leq x_2\ue(-\pi)\big\}\\
    \cup \big\{(x_1,x_2):\; x_2<0, x_1\leq -x_2\ue(\pi)\big\}.
  \end{multline*}
  Therefore, $\uE(H_\eta(0))= H_{w'}(0)\cap H_{w''}(0)$, where
  $w'=(1,\ve(\pi))$ and $w''=(1,\ue(\pi))$ for the exact dual pair
  $\ve$ and $\ue$ of nonlinear expectations with the representing set
  $\sM_\eta$. 
\end{example}

\subsection{Reduced maximal extension}
\label{sec:superl-expect}

The following result can be proved similarly to Theorem~\ref{thr:nha}
for the reduced maximal extension from \eqref{eq:9}. 

\begin{theorem}
  \label{cor:super-ha}
  A map $\ouEnot:\Lp(\sFC)\mapsto\co\sF$ is a scalarly upper
  semicontinuous normalised reduced maximal superlinear expectation if
  and only if
  \begin{equation}
    \label{eq:27a}
    \ouEnot(X)=\bigcap_{v\in\Sphere\cap\GG}
    \Big\{x:\; \langle x,v\rangle
    \leq \inf_{\gamma\in\sM_v,\E\gamma=1}\E (\gamma h(X,v))\Big\}
  \end{equation}
  for a collection of nontrivial convex $\sigma(\Lp[q],\Lp)$-closed cones
  $\sM_v\subset\Lp[q](\R_+)$, $v\in\Sphere\cap\GG$.
\end{theorem}

It is possible to take the intersection in \eqref{eq:27a} over all
$v\in\Sphere$, since $h(X,v)=\infty$ for $v\notin\GG$.  Representation
\eqref{eq:27a} can be equivalently written as the intersection of the
half-spaces $\{x:\; \langle x,v\rangle \leq \ue_v(h(X,v))\}$, where
\begin{equation}
  \label{eq:8}
  \ue_v(\beta)=\inf_{\gamma\in\sM_v,\E\gamma=1}\E (\gamma\beta)
\end{equation}
is a superlinear univariate expectation of $\beta\in\Lp(\R)$ for each
$v\in\Sphere\cap\GG$.  
The superlinear expectation \eqref{eq:27a} is law invariant if the families
$\sM_v$ are law-complete for all $v$.

\begin{corollary}
  \label{cor:rm-dil}
  Let $\ouEnot:\Lp(\sFC)\mapsto\co\sF$ be a scalarly upper
  semicontinuous law invariant normalised reduced maximal superlinear
  expectation, and let the probability space be non-atomic. Then
  $\ouEnot$ is dilatation monotonic, meaning that
  \begin{displaymath}
    \ouEnot(X)\subset \ouEnot(\E(X|\ssalg))
  \end{displaymath}
  for each sub-$\sigma$-algebra $\ssalg\subset\salg$ and all
  $X\in\Lp(\sFC)$. In particular, $\ouEnot(X)\subset \E X$.
\end{corollary}
\begin{proof}
  Since $\sM_u$ is law-complete, $\ue_v(\beta)$ given by \eqref{eq:8} is
  a law invariant concave function of $\beta\in\Lp(\R)$. Thus, $\ue_v$
  is dilatation monotonic by \cite[Cor.~4.59]{foel:sch04}, meaning
  that $\ue(\E(\xi|\ssalg))\geq \ue(\xi)$. Hence,
  \begin{displaymath}
    \ue_v(h(X,v))\leq \ue_v(\E(h(X,v)|\ssalg))
    =\ue_v(h(\E(X|\ssalg),v)). \qedhere
  \end{displaymath}
\end{proof}

\begin{example}
  \label{ex:z-eta=z}
  If $\sM_v=\sM$ in \eqref{eq:27a} is nontrivial and does not depend
  on $v$, then \eqref{eq:27a} turns into
  \begin{displaymath}
    \ouEnot(X)=\bigcap_{v\in\Sphere\cap\GG}
    \big\{x:\; \langle x,v\rangle\leq \ue(h(X,v))\big\},
  \end{displaymath}
  where $\ue$ given by \eqref{eq:8} 
  is the numerical superlinear expectation with the representing set
  $\sM$.  In this case, $\ouEnot(X)$ is the largest convex set whose
  support function is dominated by $\ue(h(X,v))$, that is,
    \begin{equation}
    \label{eq:6}
    h(\ouEnot(X),v)\leq \ue(h(X,v)), \quad v\in\GG. 
  \end{equation}
  Note that $\ue(h(X,\cdot))$ may fail to be a support function. Since
  \begin{displaymath}
    \bigcap_{v\in\Sphere\cap\GG}\big\{x:\; \langle x,v\rangle\leq \E(\gamma
    h(X,v))\}=\E(\gamma X)
  \end{displaymath}
  for $X\in\Lp(\sFC)$, this reduced maximal superlinear expectation
  admits the equivalent representation as
  \begin{equation}
    \label{eq:7} 
    \ouEnot(X)=\bigcap_{\gamma\in\sM,\E\gamma=1} \E(\gamma X).
  \end{equation}
\end{example}  

\begin{example}
  \label{ex:riesz}
  Let $X=\xi+\KK$ for a $\xi\in\Lp(\R^d)$ and a deterministic convex
  closed cone $\KK$ that is different from the whole space. Then
  \begin{equation}
    \label{eq:25}
    \ouEnot(\xi+\KK)=
    \bigcap_{v\in\Sphere\cap \GG}
    \big\{x:\; \langle x,v\rangle \leq
    \ue_v(\langle\xi,v\rangle)\big\}.
  \end{equation}
  If $\sM_v=\sM$ for all $v\in\Sphere\cap \GG$, then $\ue_v=\ue$ and 
  \begin{displaymath}
    \ouEnot(\xi+\KK)
    =\bigcap_{\gamma\in\sM,\E\gamma=1}\big(\E(\gamma\xi)+\KK).
  \end{displaymath}
  If $\KK$ is a Riesz cone, then $\ouEnot(\xi+\KK)=x+\KK$ for some
  $x$, since an intersection of translations of $\KK$ is again a
  translation of $\KK$, see \cite[Th.~26.11]{lux:zaan71}.
\end{example}

\begin{example}
  \label{ex:intersection}
  Let $\uE(X)=\E(X_1\cap \cdots\cap
  X_n)$ for $n$ independent copies of $X$, noticing that the
  expectation is empty if the intersection $X_1\cap\cdots\cap X_n$ is
  empty with positive probability. This superlinear
  expectation is neither maximal, nor even a reduced maximal one. For
  instance,
  \begin{displaymath}
    \uE(H_v(X))=H_v\big(\E\min(h(X_i,v), i=1,\dots,n)\big),
  \end{displaymath}
  so that the reduced maximal extension $\ouEnot(X)$ is the largest
  convex set whose support function is dominated by $\uE(H_v(X))$,
  $v\in\Sphere$. However, the support function of $\E(X_1\cap
  \cdots\cap X_n)$ is the expectation of the largest sublinear
  function dominated by $\min(h(X_i,v), i=1,\dots,N)$, and so $\uE(X)$
  may be a strict subset of $\ouEnot(X)$.

  For instance, let $X=\xi+\R_-^d$ for $\xi\in\Lp(\R^d)$. Then
  \begin{displaymath}
    \uE(X)=\E \min(\xi_1,\dots,\xi_n)+\R_-^d,
  \end{displaymath}
  where the minimum is applied coordinatewisely to independent copies
  of $\xi$, while $\ouEnot(X)$ is the largest convex set whose support
  function is dominated by $\E\min(\langle
  \xi_i,v\rangle,i=1,\dots,n)$, $v\in\R_+^d$. Obviously,
  \begin{displaymath}
    \min(\langle \xi_i,v\rangle,i=1,\dots,n)
    \geq \langle\min(\xi_1,\dots,\xi_n),v\rangle
  \end{displaymath}
  with a possibly strict inequality. 
\end{example}

\subsection{Minimal extension of a superlinear expectation}
\label{sec:select-superl-expect}

In any nontrivial case, the superlinear expectation of a
nondeterministic singleton is empty.  Indeed, if $\xi\in\Lp(\R^d)$,
then \eqref{eq:27a} yields that
\begin{displaymath}
  \uE(\{\xi\})\subset \ouEnot(\{\xi\})\subset \bigcap_{v\in\Sphere}
  \big\{x:\; \langle x,v\rangle
  \leq \inf_{\gamma\in\sM_v,\E\gamma=1}\E\langle \xi,\gamma v\rangle\big\}, 
\end{displaymath}
which is not empty only if 
\begin{displaymath}
  \sup_{\gamma\in\sM_{-v},\E\gamma=1}\E\langle \xi,\gamma v\rangle\leq
  \inf_{\gamma\in\sM_v,\E\gamma=1}\E\langle \xi,\gamma v\rangle
\end{displaymath}
for all $v\in\Sphere$.  In the setting of Example~\ref{ex:z-eta=z},
$\uE(\{\xi\})$ is empty unless $\ue(\langle \xi,v\rangle)+\ue(-\langle
\xi,v\rangle)\geq 0$ for all $u$. The latter means that $\ue(\langle
\xi,v\rangle)=\ve(\langle \xi,v\rangle)$ for the exact dual pair of
real-valued nonlinear expectations. Equivalently,
$\uE(\{\xi\})=\emptyset$ if $\E(\gamma\xi)\neq\E(\gamma'\xi)$ for
some $\gamma,\gamma'\in \sM$. If this is the case for all $\xi\in\Lp(X)$,
then the minimal extension of $\uE(X)$ is the set $F_X$ of fixed
points of $X$, see Example~\ref{ex:fp-nonlinear}.  Thus, it is not
feasible to come up with a nontrivial minimal extension of the
superlinear expectation if $\cone=\{0\}$.

A possible way to ensure non-emptiness of the minimal extension of
$\uE(X)$ is to apply it to random sets from $\Lp(\sFC)$ with a cone
$\cone$ having interior points, since then at least one of $h(X,v)$ and
$h(X,-v)$ is almost surely infinite for all $v\in\Sphere$. The minimal
extension of $\uE$ is given by
\begin{equation}
  \label{eq:33}
  \uuE(X)=\clo \bigcup_{\xi\in\Lp(X)}\uE(\xi+\cone).
\end{equation}
The following result, in particular, implies that the union on the
right-hand side of \eqref{eq:33} is a convex set, cf. \eqref{eq:13}. 

\begin{theorem}
  \label{thr:selection-super}
  The function $\uuE$ given by \eqref{eq:33} is a superlinear
  expectation. If \;$\uE$ in \eqref{eq:33} is reduced maximal and
  satisfies the conditions of Corollary~\ref{cor:rm-dil}, then its
  minimal extension \,$\uuE$ is law invariant and dilatation monotonic.
\end{theorem}
\begin{proof}
  Let $x$ and $x'$ belong to the union on the right-hand side of
  \eqref{eq:33} (without closure). Then $x\in\uE(\xi+\cone)$ and
  $x'\in\uE(\xi'+\cone)$, and the superlinearity of $\uE$ yields that
  \begin{displaymath}
    tx+(1-t)x'\in t\uE(\xi+\cone)+(1-t)\uE(\xi'+\cone)
    \subset \uE(t\xi+(1-t)\xi'+\cone)
  \end{displaymath}
  for each $t\in[0,1]$. Since $t\xi+(1-t)\xi'$ is a selection of $X$,
  the convexity of $\uuE(X)$ easily follows.

  The additivity on deterministic singletons, monotonicity and
  homogeneity properties are evident from
  \eqref{eq:33}. If $F\in\sFC$ is deterministic, then
  \begin{displaymath}
    \uuE(F)\subset \clo\bigcup_{x\in F}\uE(x+\cone)
    \subset \clo\bigcup_{x\in F}(x+\cone)=F.
  \end{displaymath}
  For the superadditivity property, consider $x$ and $y$ from the
  nonclosed right-hand side of \eqref{eq:33} for $X$ and $Y$,
  respectively. Then $x\in\uE(\xi+\cone)$ and $y\in\uE(\eta+\cone)$
  for some $\xi\in\Lp(X)$ and $\eta\in\Lp(Y)$. Hence,
  \begin{displaymath}
    x+y\in \uE(\xi+\cone)+\uE(\eta+\cone)\subset 
    \uE(\xi+\eta+\cone)\subset \uuE(X+Y).
  \end{displaymath}

  Now assume that $\uE$ is reduced maximal. Let $\salg_X$ be the
  $\sigma$-algebra generated by $X$. The convexity of $X$ implies that
  $\E(\xi|\salg_X)$ is a selection of $X$ for any
  $\xi\in\Lp(X)$. By the dilatation monotonicity property
  from Corollary~\ref{cor:rm-dil}, it is possible to
  replace $\xi\in\Lp(X)$ in \eqref{eq:33} with the
  family of $\salg_X$-measurable $p$-integrable selections of
  $X$. These families coincide for two identically distributed sets,
  see \cite[Prop.~1.4.5]{mo1}. The dilatation monotonicity
  $\uuE(X)\subset \uuE(\E(X|\salg))$ follows from Corollary~\ref{cor:rm-dil}.
\end{proof}

Below we establish the upper semicontinuity of the minimal extension.

\begin{theorem} 
  \label{thr:fatou}
  Assume that $p\in(1,\infty]$, $\uE$ is upper semicontinuous, and
  that $0\notin\uE(\xi+\cone)$ for all nontrivial
  $\xi\in\Lp(\cone)$. Then the minimal extension $\uuE$ is scalarly
  upper semicontinuous.
\end{theorem}
\begin{proof}
  It suffices to omit the closure in \eqref{eq:33} and
  consider $x_n\in\uE(X_n)$ such that $x_n\to x$ and $X_n\to X$
  scalarly in $\sigma(\Lp,\Lp[q])$. For each $n\geq1$, there exists a
  $\xi_n\in\Lp(X_n)$ such that $x_n\in\uE(\xi_n+\cone)$.

  Assume first that $p\in(1,\infty)$ and
  $\sup_n\E\|\xi_n\|^p<\infty$. Then $\{\xi_n,n\geq1\}$ is relatively
  compact in $\sigma(\Lp,\Lp[q])$. Without loss of generality, assume
  that $\xi_n\to\xi$. Then $\langle\xi_n,\zeta\rangle\leq
  h(X_n,\zeta)$ for all $\zeta\in\Lp[q](\GG)$. Taking expectation,
  letting $n\to\infty$ and using the convergence $\xi_n\to\xi$ and
  $X_n\to X$ yield that $\E h(\xi,\zeta)\leq \E h(X,\zeta)$. By
  Lemma~\ref{lemma:sel-domination}, $\xi$ is a selection of
  $X$. By the upper semicontinuity of $\uE$,
  \begin{displaymath}
    \limsup \uE(\xi_n+\cone)\subset \uE(\xi+\cone).
  \end{displaymath}
  Hence, $x\in \uE(\xi+\cone)$ for some $\xi\in\Lp(X)$, so that
  $x\in\uuE(X)$. 

  Assume now that $\|\xi_n\|_p^p=\E\|\xi_n\|^p\to\infty$. Let
  $\xi'_n=\xi_n/\|\xi_n\|_p$. This sequence is bounded in the
  $\Lp$-norm, and so assume without loss of generality that
  $\xi'_n\to\xi'$ in $\sigma(\Lp,\Lp[q])$. Since
  \begin{displaymath}
    x_n/\|\xi_n\|_p \in \uE((\xi_n+\cone)/\|\xi_n\|_p)
    =\uE(\xi'_n+\cone),
  \end{displaymath}
  the upper semicontinuity of $\uE$ yields that
  $0\in\uE(\xi'+\cone)$. For each $\zeta\in\Lp[q](\GG)$, we have
  $\langle\xi_n,\zeta\rangle\leq h(X_n,\zeta)$. Dividing by
  $\|\xi_n\|_p$, taking expectation, and letting $n\to\infty$ yield
  that $\E\langle\xi',\zeta\rangle\leq 0$. Thus, $\xi'\in\cone$ almost
  surely. Given that $\E\|\xi'\|=1$, this contradicts 
  the fact that $\uE(\xi'+\cone)$ contains the origin.  

  Similar reasons apply if $p=\infty$, splitting the cases when
  $\sup_n \|\xi_n\|$ is essentially bounded and when the essential
  supremum of $\|x_n\|$ converges to infinity.
\end{proof}

The exact calculation of $\uuE(X)$ involves working with all
$p$-integrable selections of $X$, which is a very rich family even in
simple cases, like $X=\xi+\KK$.  Since
\begin{equation}
  \label{eq:19}
  \uuE(X)\subset \uE(X),
\end{equation}
the superlinear expectation $\uE(X)$ yields a computationally
tractable upper bound on $\uuE(X)$.

\begin{example}
  \label{ex:reduced-family}
  Assume that $X=\xi+F$ for $\xi\in\Lp(\R^d)$ and a deterministic
  convex closed lower set $F$. Assume that $\uE$ in \eqref{eq:33} is
  reduced maximal and satisfies conditions of
  Corollary~\ref{cor:rm-dil}. Then
  \begin{equation}
    \label{eq:9s}
    \uuE(X)=\bigcup_{\xi'\in\Lp(F,\salg_\xi)} \uE(\xi+\xi'+\cone),
  \end{equation}
  where $\Lp(F,\salg_\xi)$ is the family of selections of $F$ which
  are measurable with respect to the $\sigma$-algebra generated by
  $\xi$. Indeed, $\uE(\xi+\xi'+\cone)$ is a subset of
  $\uE(\xi+\E(\xi'|\salg_\xi)+\cone)$ by Corollary~\ref{cor:rm-dil}.
\end{example}

Note that the minimal extension $\uuE$ is not necessarily a maximal
superlinear expectation. The following result describes its maximal
extension.

\begin{theorem}
  \label{thr:sel-ha}
  Assume that $\uuE$ is defined by \eqref{eq:33}, where $\uE=\ouEnot$
  is a scalarly upper semicontinuous reduced maximal superlinear
  expectation with representation \eqref{eq:7}. Then
  $\uuE(H_v(\beta))=\uE(H_v(\beta))$ for all $v\in\Sphere\cap\GG$ and
  $\beta\in\Lp(\R)$, and the reduced maximal extension of $\uuE$
  coincides with $\uE$.
\end{theorem}
\begin{proof}
  By \eqref{eq:27a}, $\uE(H_v(\beta))=H_v(\ue(\beta))$.  
  In view of \eqref{eq:19}, it suffices to show that each $x\in
  H_v(\ue(\beta))$ also belongs to $\uuE(H_v(\beta))$.  Let $y$ be the
  projection of $x$ onto the subspace orthogonal to $v$. It suffices
  to show that $x-y\in \uuE(H_v(\beta)-y)$. Noticing that
  $H_v(\beta)-y=H_v(\beta)$, it is possible to assume that $x=tv$ for
  $t\leq\ue(\beta)$. 
  
  Consider $\xi=\beta v$. Then
  \begin{displaymath}
    \uE(\beta v+\cone)=\bigcap_{w\in\GG} H_w(\ue(\langle \beta
    v,w\rangle))
    =\bigcap_{w\in\GG} H_w(\langle v,w\rangle \ue(\beta)).
  \end{displaymath}
  Since $\langle tv,w\rangle\leq \langle v,w\rangle \ue(\beta)$, we
  deduce that $x\in\uE(\xi+\cone)\subset\uuE(H_v(\beta))$. 
  
  Since $\uuE$ and $\uE$ coincide on half-spaces, the reduced maximal
  extension of $\uuE$ is
  \begin{align*}
    \underline{\ouEnot}(X)&=\bigcap_{v\in\Sphere\cap\GG} \uuE(\HvX)\\
    &=\bigcap_{v\in\Sphere\cap\GG} \uE(\HvX)=\ouEnot(X)=\uE(X). \qedhere
  \end{align*}
\end{proof}

In view of \eqref{eq:19}, $\uE(X)=\uuE(X)$ if 
\begin{equation}
  \label{eq:61}
  h(\uE(X),v)\leq \sup_{\xi\in\Lp(X)} \langle \uev(\xi),v\rangle, 
  \quad v\in\GG. 
\end{equation}
This surely holds for $X=\xi+\cone$ and also for $X$ being a
half-space with a deterministic normal. In general, $\uuE(X)$ may be a
strict subset of $\uE(X)$ as the following example shows, so
superlinear expectations are not exact even on rather simple random
sets of the type $\xi+\cone$.

\begin{example}
  \label{ex:two-values}
  Assume that $\cone=\R_-^2$ and consider 
  $\xi\in\R^2$ which equally takes two possible values: the origin and
  $a=(a_1,a_2)$. Let $X=\xi+\mathbb{K}$, where $\mathbb{K}$ is the cone containing
  $\R_-^2$ and with points $(1,-\pi)$ and $(-\pi',1)$ on its boundary,
  such that $\pi,\pi'>1$. 

  Let $\sM_v=\sM$ be the family from Example~\ref{ex:avar} and let
  $\ue$ be the corresponding superlinear expectation with the
  representing set $\sM$. For each
  $\beta\in\Lp[1](\R)$, $\ue(\beta)$ equals the
  average of the $t$-quantiles of $\beta$ over $t\in(0,\alpha)$.  If
  $\alpha\in(0,1/2]$ and $\beta$ takes two values with equal
  probabilities, then $\ue(\beta)$ is the smaller value of
  $\beta$. Then $\uuE(X)=\mathbb{K}\cap(a+\mathbb{K})$, so that $\uuE(X)$ coincides
  with $\ouEnot(X)$ in this case, see Example~\ref{ex:intersection}.
  
  Now assume that $\alpha\in(1/2,1)$. If $\beta$ equally likely takes
  two values $t$ and $s$, then
  $\ue(\beta)=\max(t,s)-|t-s|/(2\alpha)$, and
  \begin{displaymath}
    \ue(\langle \xi,v\rangle)=\max(\langle a,v\rangle,0)
    -\frac{1}{2\alpha}|\langle a,v\rangle|
  \end{displaymath}
  for all $v$ from $\GG=\mathbb{K}^o$. Since $\mathbb{K}$ is a Riesz
  cone, $\ouEnot(\xi+\mathbb{K})=x+\mathbb{K}$ for some $x$, see Example~\ref{ex:riesz}.
  For $v\in\GG$, the linear function $\langle x,v\rangle$ is dominated
  by $\frac{1}{2\alpha}\langle a,v\rangle$ if $\langle a,v\rangle<0$
  and by $(1-\frac{1}{2\alpha})\langle a,v\rangle$ otherwise. By an
  elementary calculation, 
  \begin{displaymath}
    x=\frac{1}{2\alpha}a+\Big(\frac{1}{\alpha}-1\Big)\frac{a_1\pi'+a_2}{\pi\pi'-1}
    (-\pi',1).
  \end{displaymath}
  In view of Example~\ref{ex:reduced-family},
  it suffices to consider selections of $\mathbb{K}$ measurable with respect
  to the $\sigma$-algebra $\salg_\xi$ generated by $\xi$; these
  selections take two values from the boundary of $\mathbb{K}$ with equal
  probabilities. 
  The minimal extension $\uuE(X)$ can be found by \eqref{eq:9s},
  letting $\xi'$ equally likely take two values $y=(y_1,y_2)$ and
  $z=(z_1,z_2)$ on the boundary $\partial\mathbb{K}$ of $\mathbb{K}$. Then
  \begin{displaymath}
    h(\uuE(X),v)=\sup_{y,z\in\partial \mathbb{K}}
    \sum_{i=1}^2 (\max(y_i,a_i+z_i)-\frac{1}{2\alpha}|a_i+z_i-y_i|)v_i.
  \end{displaymath}
  Figure~\ref{fig:1} shows $\ouEnot(X)$ and $\uuE(X)$ for $\pi=\pi'=2$,
  $a=(1,-1)$, and $\alpha=0.7$. It shows that the minimal extension
  may be indeed a strict subset of the reduced maximal superlinear
  expectation. 

  \begin{figure}[htbp]
    \centering
    \includegraphics[scale=0.4]{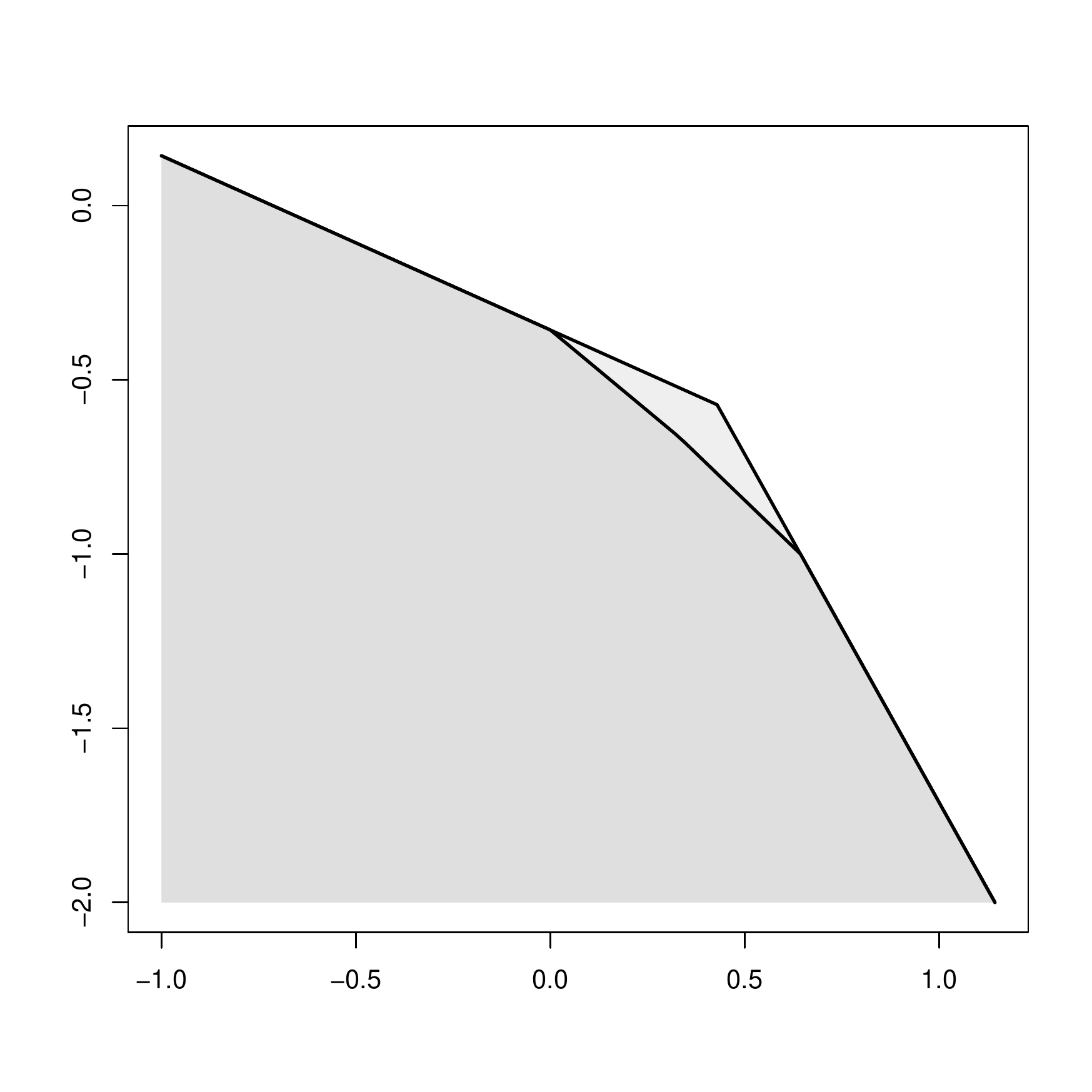}
    \caption{The reduced maximal superlinear expectation $\ouEnot(X)$
      (the larger cone) and the minimal extension $\uuE(X)$ (the
      smaller shaded set) for $X=\xi+\mathbb{K}$.    
      \label{fig:1}}
  \end{figure}
\end{example}

\section{Applications}
\label{sec:applications}

\subsection{Depth-trimmed regions and outliers}
\label{sec:subl-depth-trimm}

Consider a sublinear expectation $\vE$ restricted to the family of
$p$-integrable singletons, and let $\cone=\{0\}$. The map
$\xi\mapsto\vE(\{\xi\})$ satisfies the properties of
\emph{depth-trimmed regions} imposed in \cite{cas10}, which are those
from \cite{zuo:ser00a} augmented by the monotonicity and
subadditivity.

Therefore, the sublinear expectation provides a rather generic
construction of a depth-trimmed region associated with a random vector
$\xi\in\Lp(\R^d)$.  In statistical applications, points outside
$\vE(\{\xi\})$ or its empirical variant are regarded as
\emph{outliers}.  The subadditivity property \eqref{eq:sub-X} means
that, if a point is not an outlier for the convolution of two samples,
then there is a way to obtain this point as the sum of two
non-outliers for the original samples. 

\begin{example}[Zonoid-trimmed regions]
  \label{ex:uE-singletons}
  Fix $\alpha\in(0,1)$. For $\beta\in\Lp[1](\R)$, define
  \begin{displaymath}
    \ve_\alpha(\beta)=\alpha^{-1} \int_{1-\alpha}^1 q_\beta(s) \,
    ds, 
  \end{displaymath}
  where $q_\beta(s)$ is an $s$-quantile of $\beta$ (in case of
  non-uniqueness, the choice of a particular quantile does not matter
  because of integration).  The risk measure
  $r(\beta)=\ve_\alpha(-\beta)$ is called the \emph{average
    value-at-risk}. Denote by $\vE_\alpha$ the corresponding minimal
  sublinear expectation constructed by \eqref{eq:4}, so that
  $h(\vE_\alpha(\{\xi\}),u)=\ve_\alpha(\langle\xi,u\rangle)$ for all
  $u$.  The set $\vE_\alpha(\{\xi\})$ is the zonoid-trimmed region of
  $\xi$ at level $\alpha$, see \cite{cas10} and \cite{mos02}. This set
  can be obtained as
  \begin{displaymath}
    \vE_\alpha(\{\xi\})=\clo\big\{\E(\gamma\xi):\; \gamma\in\sP_\alpha\big\},
  \end{displaymath}
  where $\sP_\alpha\subset\Lp[1](\R_+)$ consists of all random
  variables with values in $[0,\alpha^{-1}]$ and expectation $1$, see
  Example~\ref{ex:avar}. 
  This setting is a special case of Theorem~\ref{thr:sub-dual} with
  $\sM=\{t\gamma:\; \gamma\in\sP_\alpha,t\geq0\}$. 
  The value of $\alpha$ controls the size of the depth-trimmed region,
  $\alpha=1$ yields a single point, being the
  expectation of $\xi$. 
  The subadditivity property of zonoid-trimmed regions was first
  noticed by \cite{cas:mol07}. 
\end{example}

\begin{example}[Lift expectation]
  \label{ex:lift-exp}
  Let $X$ be an integrable random closed convex set. Consider the
  random set $Y$ in $\R^{d+1}$ given by the convex hull of the origin
  and $\{1\}\times X$. The selection expectation $Z_X=\E Y$ is called
  the \emph{lift expectation} of $X$, see \cite{diay:kos:mol18}. If
  $X=\{\xi\}$ is a singleton, then $Z_X$ is the \emph{lift zonoid} of
  $\xi$, see \cite{mos02}.  By definition of the selection
  expectation, $Z_X$ is the closure of the set of
  $(\E(\beta),\E(\beta\xi))$, where $\beta$ runs through the family of
  random variables with values in $[0,1]$. Equivalently, $(\alpha,x)$
  belongs to $Z_X$ if and only if $x=\alpha\E(\gamma\xi)$ for
  $\gamma$ from the family $\sP_\alpha$, see
  Example~\ref{ex:uE-singletons}. Thus, the minimal extension
  $\uvE_\alpha$ of $\vE_\alpha$ from Example~\ref{ex:uE-singletons}
  is
  \begin{displaymath}
    \uvE_\alpha(X)=\alpha^{-1}\{x:\; (\alpha,x)\in Z_X\}.
  \end{displaymath}  
\end{example}


\subsection{Parametric families of nonlinear expectations}
\label{sec:param-famil-nonl}

Consider a dual pair $\uE$ and $\vE$ of nonlinear expectations such
that $\uE(X)\subset\E X\subset\vE(X)$ for all random closed sets
$X\in\Lp(\sFC)$. Then it is natural to regard observations of $X$ that
do not lie between the superlinear and sublinear expectation as
outliers. For each $F\in\co\sF$, it is possible to quantify its depth
with respect to the distribution of $X$ using parametric families of
nonlinear expectations constructed as follows.

Let $X_1,\dots,X_n$ be independent copies of a $p$-integrable random
closed convex set $X$. For a sublinear expectation $\vE$, 
\begin{displaymath}
  \vE_n(X)=\vE (\co(X_1\cup\cdots\cup X_n)) 
\end{displaymath}
is also a sublinear expectation. The only slightly nontrivial
property is the subadditivity, which follows from the fact that
\begin{displaymath}
  (X_1+Y_1)\cup\cdots\cup(X_n+Y_n)
  \subset (X_1\cup\cdots\cup X_n)+(Y_1\cup\cdots\cup Y_n). 
\end{displaymath}
If $X_1\cap\cdots\cap X_n$ is a.s. non-empty, then
\begin{displaymath}
  \uE_n(X)=\uE(X_1\cap\cdots\cap X_n)
\end{displaymath}
yields a superlinear expectation, noticing that 
\begin{displaymath}
  (X_1+Y_1)\cap\cdots\cap(X_n+Y_n)
  \supset (X_1\cap\cdots\cap X_n)+(Y_1\cap\cdots\cap Y_n). 
\end{displaymath}

It is possible to consistently let $\uE_\lambda^\cap(X)=\emptyset$ if
$X_1\cap\cdots\cap X_{N}$ is empty with positive probability. 

\begin{proposition}
  \label{prop:parametric}
  Let $N$ be a geometric random variable such that, for some
  $\lambda\in(0,1]$, $\Prob{N=k}=\lambda(1-\lambda)^{k-1}$, $k\geq1$,
  which is independent of $X_1,X_2,\dots$, being i.i.d. copies of
  $X$. Then
  \begin{equation}
    \label{eq:36}
    \vE_\lambda^\cup(X)=\vE (\co(X_1\cup\cdots\cup X_{N}))
  \end{equation}
  is a sublinear expectation and, if $X_1\cap\cdots\cap
  X_n\neq\emptyset$ a.s. for all $n$, then
  \begin{equation}
    \label{eq:35}
    \uE_\lambda^\cap(X)=\uE(X_1\cap\cdots\cap X_{N})
  \end{equation}
  is a superlinear expectation depending on $\lambda\in(0,1]$.
\end{proposition}

\begin{example}
  \label{ex:iter-exp}
  Choosing $\vE(X)=\uE(X)=\E X$ in \eqref{eq:36} and \eqref{eq:35}
  yields a family of nonlinear expectations depending on parameter,
  which are also easy to compute. 
\end{example}

It is easily seen that $\vE_\lambda^\cup(X)$ increases and
$\uE_\lambda^\cap(X)$ decreases as $\lambda$ declines.  Define the
depth of $F\in\sFC$ as
\begin{displaymath}
  \mathrm{depth}(F)=\sup\{\lambda\in(0,1]:\;  
  \uE_\lambda^\cap(X)\subset F\subset \vE_\lambda^\cup(X)\}.
\end{displaymath}
It is easy to see that $\vE_1^\cup(X)=\vE(X)$,
$\uE_1^\cap(X)=\uE(X)$. Furthermore, $\uE_\lambda^\cap(X)$ declines to
the set of fixed points of $X$ and $\vE_\lambda^\cup(X)$ increases to
the support of $X$ as $\lambda\downarrow 0$, see
Example~\ref{ex:fp-nonlinear}.  Thus, all closed convex sets $F$
satisfying $F_X\subset F\subset \supp X$ have a positive depth.

In order to handle the empirical variant of this concept based on a
sample $X_1,\dots,X_n$ of independent observations of $X$, 
consider a random closed set $\tilde{X}$ that with equal probabilities
takes one of the values $X_1,\dots,X_n$. Its distribution can be
simulated by sampling one of these sets with possible
repetitions. Then it is possible to use the nonlinear expectations of
$\tilde{X}$ in order to assess the depth of any given convex set,
including those from the sample.

\subsection{Risk of a set-valued portfolio}
\label{sec:risk-set-valued}

For a random variable $\xi\in\Lp(\R)$ interpreted as a financial
outcome or gain, the value $\ve(-\xi)$ (equivalently, $-\ue(\xi)$) is
used in finance to assess the risk of $\xi$. It may be tempting to
extend this to the multivariate setting by assuming that the risk is a
$d$-dimensional function of a random vector $\xi\in\Lp(\R^d)$, with
the conventional properties extended coordinatewisely. However, in
this case the nonlinear expectations (and so the risk) are
marginalised, that is, the risk of $\xi$ splits into a vector of
nonlinear expectations applied to the individual components of $\xi$,
see Theorem~\ref{thr:split}.

Moreover, assessing the financial risk of a vector $\xi$ is impossible
without taking into account exchange rules that can be applied to its
components. If no exchanges are allowed and only consumption is
possible, then one arrives at positions being selections of
$X=\xi+\R_-^d$. On the contrary, if the components of $\xi$ are
expressed in the same currency with unrestricted exchanges and
disposal (consumption) of the assets, then each position from the
half-space $X=\{x:\; \sum x_i\leq \sum \xi_i\}$ is reachable from
$\xi$. Working with the random set $X$ also eliminates possible
non-uniqueness in the choice of $\xi$ with identical sums.

In view of this, it is natural to consider multivariate financial
positions as lower random closed convex sets, equivalently, those from
$\Lp(\sFC)$ with $\cone=\R_-^d$.  The random closed set is said to be
acceptable if $0\in\uE(X)$, and the risk of $X$ is defined as
$-\uE(X)$. The superadditivity property guarantees that if both $X$
and $Y$ are acceptable, then $X+Y$ is acceptable. This is the
classical financial diversification advantage formulated in set-valued
terms.

If $X\in\Lp(\sFC)$ and $\cone=\R_-^d$, the minimal extension
\eqref{eq:33} is called the \emph{lower set extension} of $\uE$.
If $\uE$ is reduced maximal, \eqref{eq:7} yields that
\begin{equation}
  \label{eq:8a}
  \uE(\xi+\R_-^d)
  =\bigcap_{\gamma\in\sM,\E\gamma=1} (\E(\gamma\xi)+\R_-^d)
  =\uev(\xi)+\R_-^d, 
\end{equation}
where $\uev(\xi)=(\ue(\xi_1),\dots,\ue(\xi_d))$ is defined by applying
the same superlinear expectation $\ue$ with representing set
$\sM$ to each component of $\xi$.  Then
\begin{equation}
  \label{eq:selection-super-1}
  \uuE(X)=\clo\bigcup_{\xi\in\Lp(X)} \big(\uev(\xi)+\R_-^d\big)
\end{equation}
In other words, $\uuE(X)$ is the closure of the set of all points
dominated coordinatewisely by the superlinear expectation of at
least one selection of $X$.  In \cite{cas:mol14}, the
origin-reflected set $-\uuE(X)$ was called the selection risk
measure of $X$.  

For set-valued portfolios $X=\xi+\cone$, arising as the sum of a
singleton $\xi$ and a (possibly random) convex cone $\cone$, the
maximal superlinear expectation (in our terminology), considered a
function of $\xi$ only and not of $\xi+\cone$, was studied by
\cite{ham:hey10} and \cite{ham:hey:rud11}. The case of general
set-valued arguments was pursued by \cite{cas:mol14}.  For the purpose
of risk assessment, one can use any superlinear expectation. However,
the sensible choices are the maximal superlinear expectation in view
of its closed form dual representation, and the lower set extension in
view of its direct financial interpretation (through its primal
representation), meaning the existence of a selection with all
acceptable components.  Given that the minimal superlinear expectation
may be a strict subset of the maximal one (see
Example~\ref{ex:two-values}), the acceptability of $X$ under a maximal
superlinear expectation may be a weaker requirement than the
acceptability under the lower set extension.

\renewcommand{\thesection}{Appendix}
\renewcommand{\thetheorem}{A.\arabic{theorem}}

\section{Marginalisation of vector-valued sublinear functions}
\label{sec:append-triv-vect}

It may be tempting to consider vector-valued functions
$\vev:\Lp(\R^d)\mapsto\R^d$, which are sublinear, that is,
$\vev(x)=x$ for all $x\in\R^d$,
$\vev(\xi)\leq \vev(\eta)$ if $\xi\leq \eta$ a.s.,
$\vev(c\xi)=c\vev(\xi)$ for all $c\geq0$, and 
\begin{displaymath}
  \vev(\xi+\eta)\leq \vev(\xi)+\vev(\eta). 
\end{displaymath}
Such a function may be viewed as a restriction of a sublinear
set-valued expectation onto the family of sets $\xi+\R_-^d$ and
letting $\vev(\xi)$ be the coordinatewise supremum of
$\vE(\xi+\R_-^d)$. 

The following result shows that vector-valued sublinear expectations
marginalise, that is, they split into sublinear expectations applied
to each component of the random vector. 

\begin{theorem}
  \label{thr:split}
  If $\vev$ is a $\sigma(\Lp,\Lp[q])$-lower semicontinuous
  vector-valued sublinear expectation, then
  \begin{displaymath}
    \vev(\xi)=(\ve_1(\xi_1),\dots,\ve_d(\xi_d))
  \end{displaymath}
  for a collection of numerical sublinear expectations
  $\ve_1,\dots,\ve_d$. 
\end{theorem}
\begin{proof}
  The set $\sA=\{\xi:\; \vev(\xi)\leq 0\}$ is a
  $\sigma(\Lp,\Lp[q])$-closed convex cone in $\Lp(\R^d)$. The polar
  cone $\sA^o$ is the set of all $\R^d$-valued measures
  $\muv=(\mu_1,\dots,\mu_d)$ such that
  \begin{displaymath}
    \int \xi d\muv=\left(\int\xi_1 d\mu_1,\dots,\int \xi_d
      d\mu_d\right)\leq 0
  \end{displaymath}
  for all $\xi\in\sA$. It is easy to see that each $\muv\in\sA$ has
  all nonnegative components. The bipolar theorem yields that
  \begin{displaymath}
    \sA=\left\{\xi:\; \int \xi d\muv\leq 0 \;\text{for all}
      \; \muv\in\sA^o\right\}.
  \end{displaymath}
  Since $\vev$ is constant preserving, 
  \begin{displaymath}
    \vev(\xi+x)-x\leq \vev(\xi)=\vev((\xi+x)-x)\leq \vev(\xi+x)-x,
  \end{displaymath}
  so that $\vev(\xi+x)=\vev(\xi)+x$ for all deterministic $x\in\R^d$. 
  Hence,
  \begin{equation}
    \label{eq:12} \tag{A.1}
    \vev(\xi)=\inf \bigcap_{\muv\in\sA^o} \Big\{y\in\R^d:\; \int \xi
    d\muv \leq \int y d\muv\Big\},
  \end{equation}
  where the infimum is taken coordinatewisely.

  Consider the set $C_\muv=\{y\in\R^d:\; \int \xi d\muv \leq \int y
  d\muv\}$ for some $\muv=(\mu_1,\dots,\mu_d)\in\sA^o$. Let $\sA_i^o$
  denote the family of all nontrivial $\muv\in\sA^o$ such that
  $\mu_j$ vanish for all $j\neq i$. Note that if $\muv\in\sA^o$, then
  $(\mu_1,0,\dots,0)\in\sA_1^o$, that is the projections of $\sA^o$
  and $\sA_i^o$ on each of the component coincide. If
  $\muv\in\sA_1^o$, then
  \begin{displaymath}
    C_\muv=\Big[\int \xi_1 d\mu_1,\infty\Big)\times \R\times\cdots\times\R.
  \end{displaymath}
  Assume that two components of $\muv$ do not vanish, say $\mu_1$ and
  $\mu_2$. Then 
  \begin{align*}
    C_\muv&=\Big\{y:\; \int \xi_1 d\mu_1+\int \xi_2 d\mu_2 \leq \int y_1
    d\mu_1 + \int y_2 d\mu_2 \Big\}\\ & \supset 
    \Big[\int \xi_1 d\mu_1,\infty \Big)\times \Big[\int \xi_2 d\mu_2,\infty \Big)
    \times\R\times\cdots\times\R.
  \end{align*}
  Thus, this latter set $C_\muv$ does not influence the coordinatewise
  infimum in \eqref{eq:12} comparing to the sets obtained by letting
  $\muv\in \sA_1^o\cup\sA_2^o$. The same argument applies to $\muv\in\sA^o$
  with more than two nonvanishing components. Thus, the intersection
  in \eqref{eq:12} can be taken over
  $\muv\in\sA_1^o\cup\cdots\cup\sA_d^o$, whence the result. 
\end{proof}

A similar result holds for superlinear vector-valued expectations.

\section*{Acknowledgements}

IM is grateful to Ignacio Cascos for discussions and a collaboration
on related works. This work was motivated by the stay of IM at the
Universidad Carlos III de Madrid in 2012 supported by the Santander
Bank.  IM was also supported in part by Swiss National Science
Foundation grants 200021\_153597 and IZ73Z0\_152292.


\newcommand{\noopsort}[1]{} \newcommand{\printfirst}[2]{#1}
  \newcommand{\singleletter}[1]{#1} \newcommand{\switchargs}[2]{#2#1}

\end{document}